\documentclass[11pt]{article}
 \usepackage{amsmath, amsthm, amssymb, bbm, setspace, bigints}
 \usepackage[margin=1 in]{geometry}
\usepackage{caption}
\usepackage{subcaption}
\usepackage[toc,page]{appendix}
\usepackage[pdftex]{graphicx}
\usepackage{pdfpages}
\usepackage{epstopdf}

\usepackage{authblk}

\usepackage{caption}
\usepackage{subcaption}

\usepackage{dsfont}
\usepackage{amsmath}
\usepackage{multicol}
\usepackage{multirow}
\usepackage{mathrsfs,amsmath,amssymb,color}
\usepackage{bm}
\usepackage{bbm, dsfont}
\usepackage{natbib}

\usepackage{amsmath}
\usepackage{amsthm,bbm,amssymb}
\usepackage{amsmath}
\usepackage{natbib}
\usepackage[colorlinks,citecolor=blue,urlcolor=blue,filecolor=blue,backref=page]{hyperref}

\usepackage[utf8]{inputenc}
\usepackage{amsmath,amssymb,natbib,xcolor,multicol,url,mhequ}
\usepackage{graphicx,bbm,xr}
\usepackage{booktabs,epstopdf,color}
\usepackage[space]{grffile}
\usepackage{lineno}
\usepackage[plain,noend]{algorithm2e}
\usepackage{multirow}

\def\mr{\mathrm}
\def\ind{\mathbbm{1}}

\def\m{\mathcal}
\def\mb{\mathbb}
\def\wt{\widetilde}
\def\T{{\mathrm{\scriptscriptstyle T} }}

\def\m{\mathcal}
\def\mc{\mathcal}
\def\mb{\mathbb}
\def\mr{\mathrm}

\newcommand{\Var}{\mathop{\rm var}}

\newcommand \bbP{\mathbb{P}}
\newcommand \bbE{\mathbb{E}}

\def\ind{\mathbbm{1}}

\def\T{{ \mathrm{\scriptscriptstyle T} }}

\def\mr{\mathrm}
\def\ind{\mathbbm{1}}

\newcommand{\be}{\begin{equs}}
\newcommand{\ee}{\end{equs}}

\numberwithin{equation}{section}
\theoremstyle{plain}
\newtheorem{theorem}{Theorem}
\newtheorem{assumption}{Assumption}
\newtheorem{remark}{Remark}

\newtheorem{lemma}{Lemma}
\newtheorem{proposition}[theorem]{Proposition}

\title{Mass-shifting phenomenon of truncated multivariate normal priors}

\author[1]{Shuang Zhou\thanks{shuang@stat.tamu.edu}}
\author[1]{Pallavi Ray\thanks{pallaviray@stat.tamu.edu}}
\author[1]{Debdeep Pati\thanks{debdeep@stat.tamu.edu}}
\author[1]{Anirban Bhattacharya\thanks{anirbanb@stat.tamu.edu}}
\affil[1]{Department of Statistics,  Texas A\&M University, College Station, Texas, 77843, USA}

\begin{document}
\maketitle

\begin{abstract}
We show that lower-dimensional marginal densities of dependent zero-mean normal distributions truncated to the positive orthant exhibit a {\em mass-shifting} phenomenon. Despite the truncated multivariate normal density having a mode at the origin, the marginal density assigns increasingly small mass near the origin as the dimension increases.  The phenomenon accentuates with stronger correlation between the random variables. A precise quantification characterizing the role of the dimension as well as the dependence is provided.
This surprising behavior has serious implications towards Bayesian constrained estimation and inference, where the prior, in addition to having a full support, is required to assign a substantial probability near the origin to capture flat parts of the true function of interest. 
Without further modification, we show that truncated normal priors are not suitable for modeling flat regions and propose a novel alternative strategy based on shrinking the coordinates using a multiplicative scale parameter. The proposed shrinkage prior is empirically shown to guard against the mass shifting phenomenon while retaining computational efficiency.   

\end{abstract}

\section{Introduction}
Let $p(\cdot)$ denote the density of a $\mc N_N(\mathbf{0}, \Sigma)$ distribution truncated to the non-negative orthant in $\mb R^N$,  
\begin{align}\label{eq:basic_def}
p(\theta) \propto e^{ -\theta^\T \Sigma^{-1} \theta/2} \, \ind_{\mc C}(\theta), \quad \mc C = [0, \infty)^N :\, = \big\{ \theta \in \mb R^N ~:~ \theta_1 \ge 0, \ldots, \theta_N \ge 0 \big \}. 
\end{align}
The density $p$ is clearly unimodal with its mode at the origin. However, for certain classes of non-diagonal $\Sigma$, we surprisingly observe that the lower-dimensional marginal distributions increasingly shift mass away from the origin as $N$ increases. This observation is quantified in Theorem \ref{thm:main}, where we provide non-asymptotic estimates for marginal probabilities of events of the form $\{\theta_1 \le \delta\}$, for $\delta > 0$. En-route to the proof, we derive a novel Gaussian comparison inequality in Lemma \ref{slepian_gl_main}. An immediate implication of this mass-shifting phenomenon is that corner regions of the support $\mc C$, where a subset of the coordinates take values close to zero, increasingly become low-probability regions under $p(\cdot)$ as dimension increases. From a statistical perspective, this helps explain a paradoxical behavior in Bayesian constrained regression empirically observed in \cite{neelon2004bayesian} and \cite{curtis2009variable}, where truncated normal priors led to biased posterior inference when the underlying function had flat regions. 

A common approach towards Bayesian constrained regression expands the function in a flexible basis which facilitates representation of the functional constraints in terms of simple constraints on the coefficient space, and then specifies a prior distribution on the coefficients obeying the said constraints. In this context, the multivariate normal distribution subject to linear constraints arises as a natural conjugate prior in Gaussian models and beyond. Various basis, 
such as Bernstein polynomials \citep{curtis2009variable}, regression splines \citep{cai2007bayesian,meyer2011bayesian}, penalized spines \citep{brezger2008monotonic}, cumulative distribution functions \citep{bornkamp2009bayesian}, restricted splines \citep{shively2011nonparametric}, and compactly supported basis \citep{maatouk2017gaussian} have been employed in the literature. For numerical illustrations in this article, we shall use the formulation of \cite{maatouk2017gaussian} where various restrictions such as boundedness, monotonicity, convexity, etc were equivalently translated into non-negativity constraints on the coefficients under an appropriate basis expansion. They used a truncated normal prior as in \eqref{eq:basic_def} on the coefficients, with $\Sigma$ induced from a parent Gaussian process on the regression function; see Appendix \ref{sec:app_posterior} for more details. 

\begin{figure}[h!]
\centering
	\begin{subfigure}{.5\textwidth}
		\centering
		\includegraphics[width=0.9\textwidth]{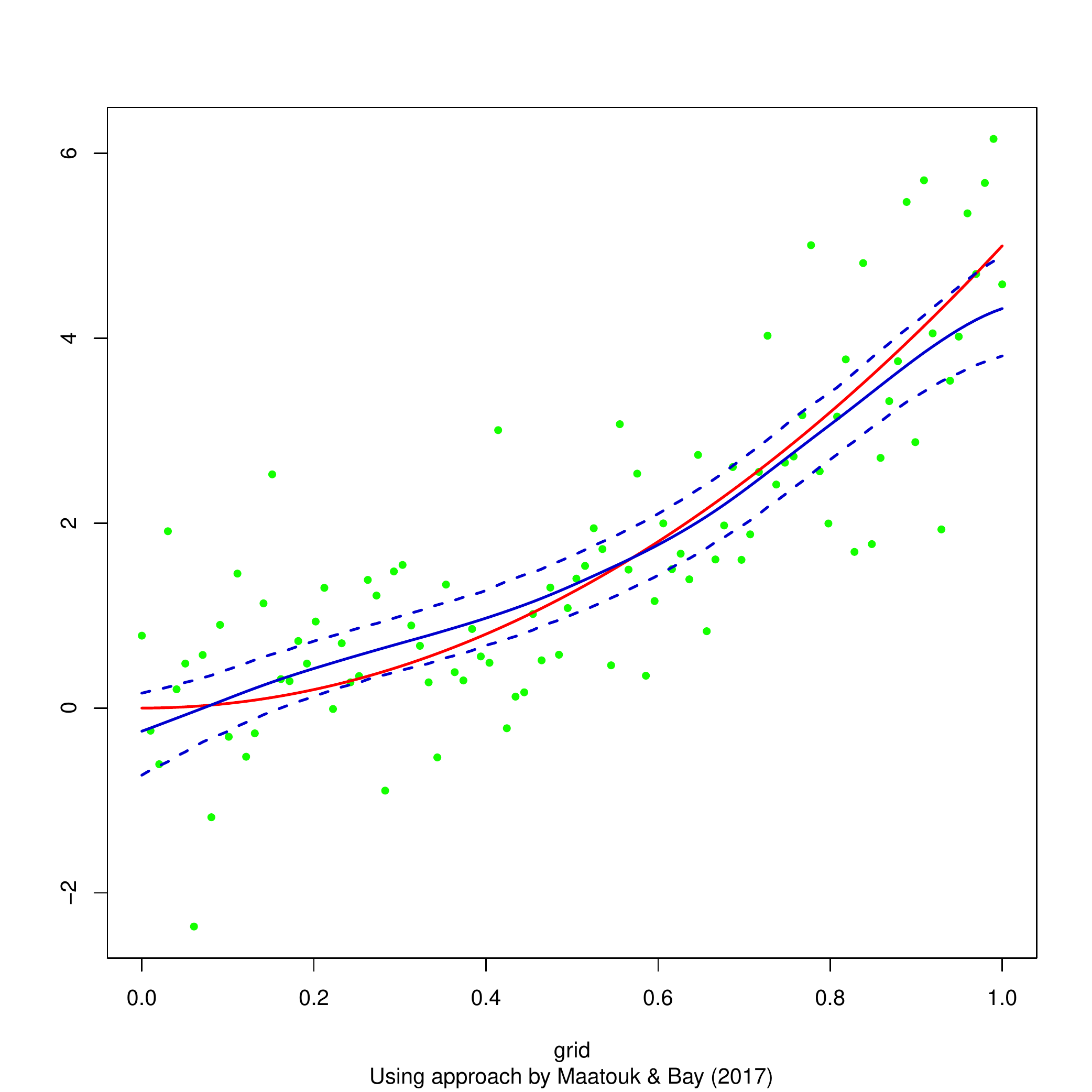}
	\end{subfigure}%
	\begin{subfigure}{.5\textwidth}
		\centering
		\includegraphics[width=0.9\textwidth]{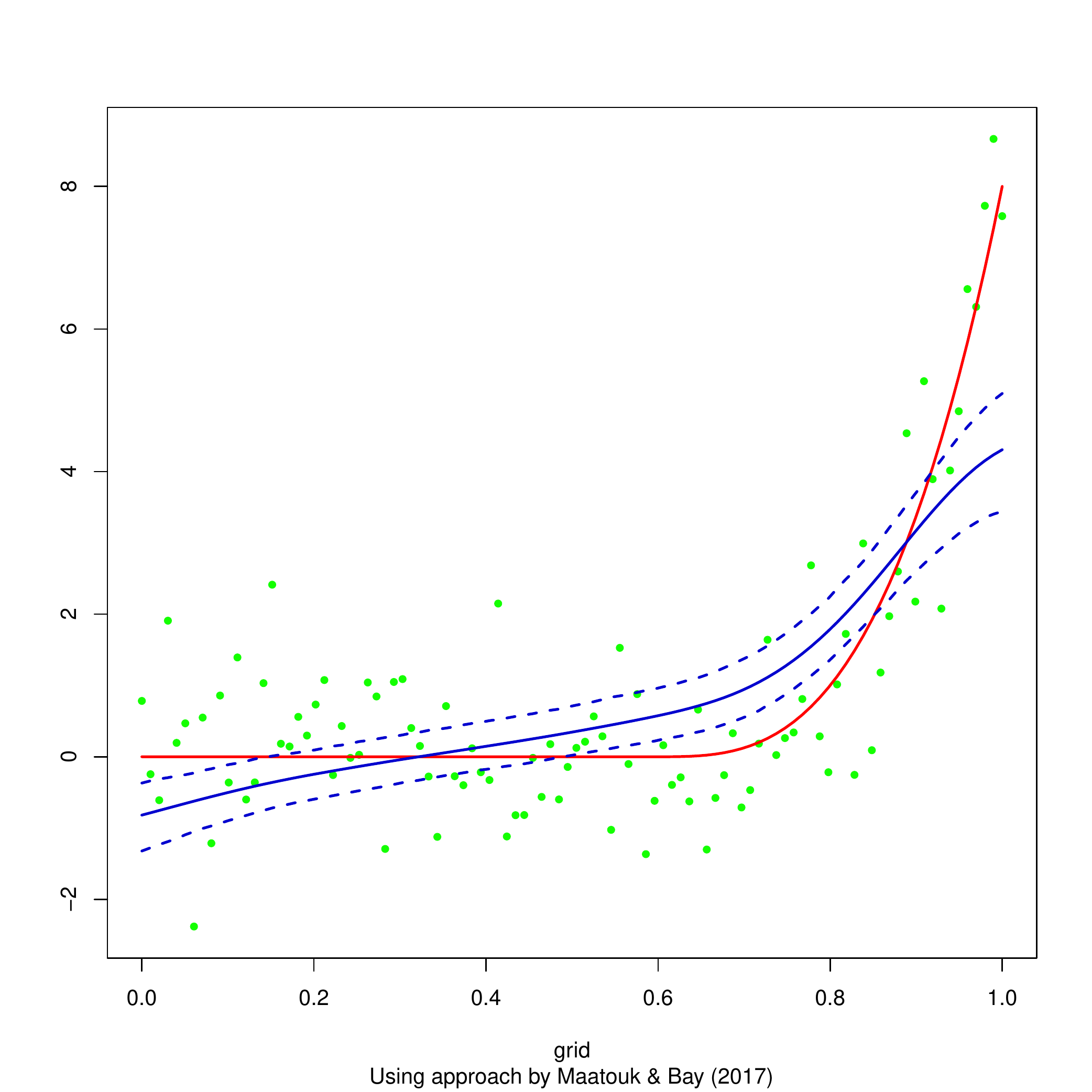}
	\end{subfigure}
	\caption{{\em Monotone function estimation using the basis of \cite{maatouk2017gaussian} and a joint truncated normal prior $p(\cdot)$ on the coefficients. Red solid curve corresponds to the true function, blue solid curve is the posterior mean, the region within two dotted blue curves represent a pointwise 95\% credible interval, and the green dots are observed data points. Left panel: true function is strictly monotone. Right panel: true function is monotone with a near-flat region.}}
	\label{fig-cgp}
\end{figure}
To motivate our theoretical investigations, the two panels in Figure \ref{fig-cgp} depict the estimation of two different monotone smooth functions on $[0, 1]$ based on 100 samples using the basis of \cite{maatouk2017gaussian} and a joint prior $p(\cdot)$ as in \eqref{eq:basic_def} on the $N = 50$ dimensional basis coefficients. The same prior scale matrix $\Sigma$ was employed across the two settings; the specifics are deferred to Section 3 and Appendix \ref{sec:app_posterior}. Observe that the function in the left panel is strictly monotone, while the one on the right panel is relatively flat over a region. While the point estimate (posterior mean) as well as the credible intervals look reasonable for the function in the left panel, the situation is significantly worse for the function in the right panel. The posterior mean incurs a large bias, and the pointwise 95$\%$ credible intervals fail to capture the true function for a substantial part of the input domain, suggesting that the entire posterior distribution is biased away from the truth. This behavior is perplexing; we are fitting a well-specified model with a prior that has full support\footnote{the prior probability assigned to arbitrarily small Kullback--Leibler neighborhoods of any point is positive.} on the parameter space, which under mild conditions implies good first-order asymptotic properties \citep{ghosal2000} such as posterior consistency. However, the finite sample behavior of the posterior under the second scenario clearly suggests otherwise. 

Functions with flat regions as in the right panel of Figure \ref{fig-cgp} routinely appear in many applications; for example, dose-response curves are assumed to be non-decreasing with the possibility that the dose-response relationship is flat over certain regions \citep{neelon2004bayesian}. A similar biased behavior of the posterior for such functions under truncated normal priors was observed by \cite{neelon2004bayesian} while using a piecewise linear model, and also by \cite{curtis2009variable} under a Bernstein polynomial basis. However, a clear explanation behind such behavior as well as the extent to which it is prevalent has been missing in the literature, and the mass-shifting phenomenon alluded before offers a clarification. Under the basis of \cite{maatouk2017gaussian}, a subset of the basis coefficients are required to shrink close to zero to accurately approximate functions with such flat regions. However, the truncated normal posterior pushes mass away from such corner regions, leading to the bias.  Importantly, our theory also suggests that the problem would not disappear and would rather get accentuated in the large sample scenario if one follows standard practice of scaling up the number of basis functions with increasing sample size, since the mass-shifting gets more pronounced with increasing dimension. To illustrate this point, Figure \ref{fig-cgp1} shows the estimation of the same function in the right panel of Figure \ref{fig-cgp}, now based on 500 samples and $N = 50$ and $N = 250$ basis functions in the left and right panel respectively. Increasing the number of basis functions indeed results in a noticeable increase in the bias as clearly seen from the insets which zoom into two disjoint regions of the covariate domain. A similar story holds for the basis of \cite{neelon2004bayesian} and \cite{curtis2009variable}. 

\begin{figure}[h!]
\centering
	\begin{subfigure}{.5\textwidth}
		\centering
		\includegraphics[width=\textwidth]{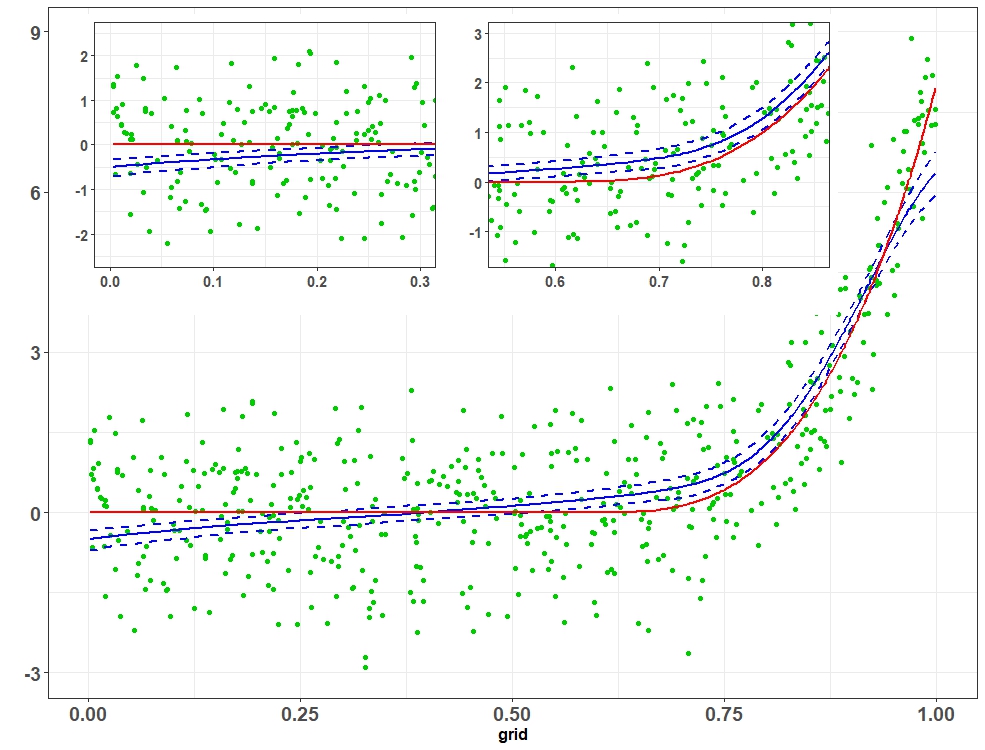}
	\end{subfigure}%
	\begin{subfigure}{.5\textwidth}
		\centering
		\includegraphics[width=\textwidth]{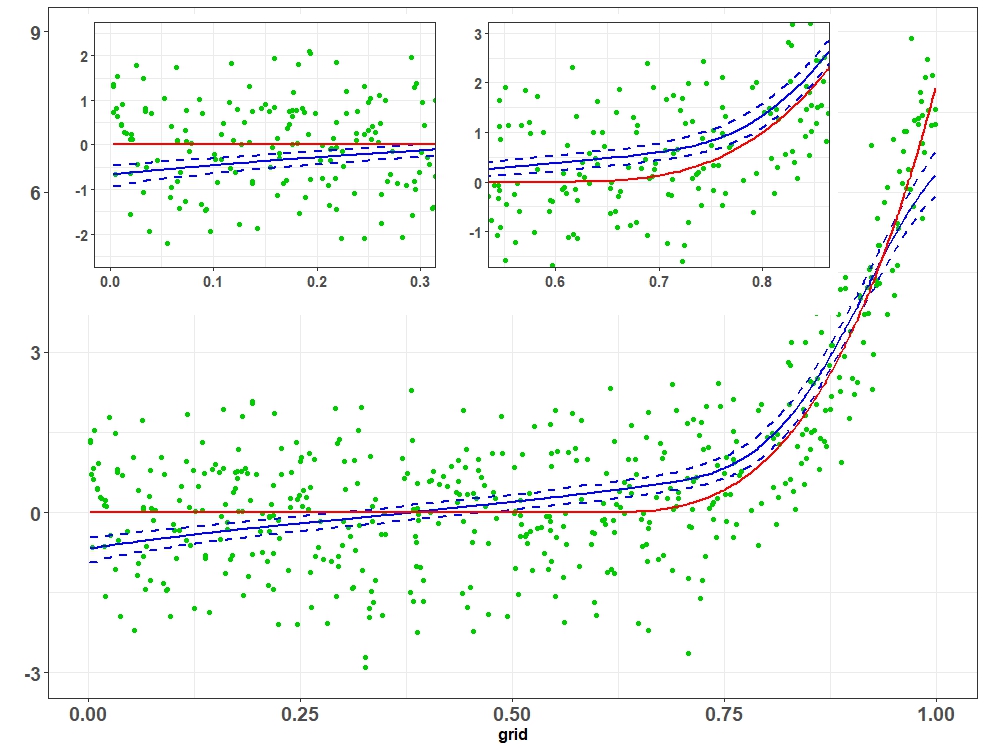}
	\end{subfigure}
	\caption{{\em Monotone function estimation using the basis of \cite{maatouk2017gaussian} and a joint truncated normal prior on the coefficients. Red solid curve corresponds to the true function, blue solid curve is the posterior mean, the region within two dotted blue curves represent a pointwise 95\% credible interval, and the green dots are observed data points corresponding to $N = 50$ (left panel) and $N = 250$ (right panel).}}
	\label{fig-cgp1}
\end{figure}

\cite{neelon2004bayesian} and \cite{curtis2009variable} both used point-mass mixture priors as remedy, which is a natural choice under a non-decreasing constraint. However, their introduction becomes somewhat cumbersome under the non-negativity constraint in \eqref{eq:basic_def}. As a simple remedy, we suggest introducing a multiplicative scale parameter for each coordinate {\em a priori} and further equipping it with a prior mixing distribution which has positive density at the origin and heavy tails; a default candidate is the half-Cauchy density \citep{carvalho2010horseshoe}. The resulting prior shrinks much more aggressively towards the origin, and we empirically illustrate its superior performance over the truncated normal prior. This empirical exercise provides further support to our argument.

\section{Mass-shifting phenomenon of truncated normal distributions} \label{sec:theory}

\subsection{Marginal densities of truncated normal distributions} 
Our main focus is on studying the properties of marginal densities of truncated normal distributions described in equation \eqref{eq:basic_def} and quantifying how they behave with increasing dimensions. We begin by introducing some notation. 
We use $\m N(\gamma, \Omega)$ to denote the $d$-dimensional normal distribution with mean $\gamma \in \mb R^d$ and positive definite covariance matrix $\Omega$; also let $\m N(x; \gamma, \Omega)$ denote its density evaluated at $x \in \mb R^d$.  We reserve the notation $\Sigma_d(\rho)$ to denote the $d \times d$ compound-symmetry correlation matrix with diagonal elements equal to $1$ and off-diagonal elements equal to $\rho \in (0,1)$,
\begin{align}\label{eq:comp_symm}
\Sigma_d(\rho) = (1 - \rho) \mr I_d + \rho \mathbf{1}_d \mathbf{1}_d^\T,
\end{align} 
with $\mathbf{1}_d$ the vector of ones in $\mb R^d$ and $\mr I_d$ the $d \times d$ identity matrix.

For a subset $\m C \subset \mb R^N$ with positive Lebesgue measure, let $\m N_{\m C}(\gamma, \Omega)$ denote a $\m N(\gamma, \Omega)$ distribution truncated onto $\m C$, with density 
\begin{align}\label{eq:tmvn_den}
\widetilde{p}(\theta) = m_{\m C}^{-1} \, \m N_N(\theta; \gamma, \Omega) \, \ind_{\m C}(\theta), 
\end{align}
where $m_{\m C} = \bbP (X \in \m C)$ for $X \sim \m N(\gamma, \Omega)$ is the constant of integration and $\ind_{\m C}(\cdot)$ the indicator function of the set $\m C$.  We throughout assume $\m C$ to be the positive orthant of $\mb R^N$ as in equation \eqref{eq:basic_def}, namely, $\m C = [0, \infty)^N$; a general $\m C$ defined by linear inequality constraints can be reduced to rectangular constraints using a linear transformation - see, for example, \S\,2 of \cite{botev2017normal}. The dimension $N$ will be typically evident from the context.


Our investigations were originally motivated by the following observation. Consider $\theta \sim \m N_{\m C}(\mathbf{0}, \Sigma_2(\rho))$ for $\rho \in (0, 1)$. Then, the marginal distribution of $\theta_1$ has density proportional to $e^{-\theta_1^2/2}\, \Phi\{\rho \theta_1/(1-\rho^2)^{1/2}\}$ on $(0, \infty)$, where $\Phi$ denotes the $\m N(0, 1)$ cumulative distribution function. This distribution is readily recognized as a skew normal density \citep{azzalini1996multivariate} truncated to $(0, \infty)$. Interestingly, the marginal of $\theta_1$ has a strictly positive mode, while the joint distribution of $\theta$ had its mode at $\mathbf{0}$. \cite{cartinhour1990one} noted that the truncated normal family is not closed under marginalization for non-diagonal $\Sigma$, and derived a general formula for the univariate marginal as the product of a univariate normal density with a skewing factor. In Proposition \ref{PROP:MARG_K} below, we generalize Cartinhour's result for any lower-dimensional marginal density. We write the scale matrix $\Sigma_N$ in block form as $\Sigma_N = [\Sigma_{k,k},\, \Sigma_{N-k,k}; \Sigma_{k,N-k},\, \Sigma_{N-k,N-k}]$. 
\begin{proposition}\label{PROP:MARG_K}
Suppose $\theta \sim \m N_\m C(\mathbf{0}_N, \Sigma_N)$. The marginal density $\widetilde{p}_{k, N}$ of $\theta^{(k)} = (\theta_1, \ldots, \theta_k)^\T$ is
\begin{align*}
\widetilde{p}_{k, N}(\theta_1, \ldots, \theta_k) 
=& (2\pi)^{-k/2}\, m_{\m C}^{-1} \, e^{-\frac{1}{2} {\theta^{(k)}}^\T \Sigma_{k,k}^{-1} \theta^{(k)}} \,  \\
\times \,& \bbP(\widetilde{X}_{N-k} \le \Sigma_{N-k,k} \,\Sigma^{-1}_{k,k}\, \theta^{(k)}) \prod_{i=1}^k \ind_{[0, \infty)}(\theta_i), 
\end{align*}
where $\widetilde{X}_{N-k} \sim \m N(\mathbf{0}_{N-k}, \widetilde{\Sigma}^{-1}_{N-k,N-k})$ with $\widetilde{\Sigma}_{N-k,N-k} = (\Sigma_{N-k,N-k} - \Sigma_{N-k,k}\,\Sigma^{-1}_{k,k}\,\Sigma_{k,N-k})^{-1}$,  and the $\le$ symbol is to be interpreted elementwise. 
Here, the constant $m_{\m C} =\bbP(X \in \m C)$ for $X \sim \m N(\mathbf{0}_N, \Sigma_N)$.
\end{proposition}
 
When $k=1$, Proposition \ref{PROP:MARG_K} implies 
\begin{align}\label{eq:univ_mar}
\widetilde{p}_{1,N} \propto e^{-\theta^2_1/(2\Sigma_{1,1})}\bbP(\widetilde{X}_{N-1}\le\Sigma_{N-1,1} \,\theta_1/\Sigma_{1,1}) \ind_{[0, \infty)}(\theta_1). 
\end{align}
Let $\m S_N$ denote the set of $N \times N$ covariance matrices whose correlation coefficients are all non-negative. The map $\theta_1 \mapsto e^{-\theta^2_1/(2\Sigma_{1,1})}$ is decreasing and when $\Sigma_N\in \m S_N$, $\theta_1 \mapsto \bbP(\widetilde{X}_{N-1}\le\Sigma_{N-1,1} \,\theta_1/\Sigma_{1,1})$ is increasing, on $(0, \infty)$. Thus, if $\Sigma_N\in \m S_N$, $\widetilde{p}_{1,N}$ is unimodal with a strictly positive mode.


As another special case, suppose $\Sigma = \Sigma_N(\rho)$ for some $\rho \in (0, 1)$ and let $k = N-1$. We then have, 
$$
\widetilde{p}_{N-1,N} \propto e^{- {\theta^{(N-1)}}^\T \, \Sigma_{N-1}^{-1}(\rho) \,\theta^{(N-1)} } \, \Phi(a^\T \theta^{(N-1)}) \ \prod_{i=1}^{N-1} \ind_{[0, \infty)}(\theta_i),
$$
with $a = C_\rho \big(\sum_{i=1}^{N-1} \theta_i \big) \, \mathbf{1}_{N-1}$, where $C_\rho$ is a positive constant. 
This density can be recognized as a multivariate skew-normal distribution \citep{azzalini1996multivariate} truncated to the non-negative orthant.

\subsection{Mass-shifting phenomenon of marginal densities}
While the results in the previous section imply that the marginal distributions shift mass away from the origin, they do not precisely characterize the severity of its prevalence. In this section, we show that under appropriate conditions, the univariate marginals assign increasingly smaller mass to a fixed neighborhood of the origin with increasing dimension. In other words, the skewing factor noted by Cartinhour begins to dominate when the ambient dimension is large. In addition to the dimension, we also quantify the amount of dependence in $\Sigma_N$ contributing to this mass-shifting. To the best of our knowledge, this has not been observed or quantified in the literature. 

We state our results for $\Sigma_N \in \m B_{N,K}$, where for $2 \le K \le N-1$, $\m B_{N, K}$ denotes the space of $K$-banded nonnegative correlation matrices,
\begin{align}\label{eq:BNK}
\m B_{N,K} = \Big\{ \Sigma_N = (\rho_{ij}) \in \m S_N \,: \, \rho_{ii} = 1 \ \forall\, i, \ \ \rho_{ij} = 0 \ \forall \  |i-j| \ge K \Big\}.
\end{align}
While our main theorem below can be proved for other dependence structures, the banded structure naturally arises in statistical applications as discussed in the next section. 

Given $\Sigma_N = (\rho_{ij}) \in \m B_{N,K}$, define $\rho_{\max} = \max_{i \ne j} \rho_{ij}$ and $\rho_{\min} = \min_{i \ne j, |i-j|\le K} \rho_{ij}$ to be the maximum and minimum correlation values within the band. For $\theta \sim \m N_{\m C}(\mathbf{0}, \Sigma_N)$ with $\m C = [0, \infty)^N$, let $\alpha_{N, \delta} = \bbP(\theta_1 \le \delta)$. 
With these definitions, we are ready to state our main theorem.
\begin{theorem}\label{thm:main}
Let $\Sigma_N \in \m B_{N, K}$ be such that $(\rho_{\min}, \rho_{\max}) \in \m Q$, where 
$$
\m Q = \bigg\{ (u, v) \in (0, 1)^2: u \le v, \ \frac{u}{2(1-u)} \ge v \bigg\}. 
$$
Then, there exists a constant $K_0$ such that whenever $K \ge K_0$, we have for any $\delta > 0$, 
$$
\alpha_{N, \delta} \le C'_{\rho_{\min},\rho_{\max}} \, \delta \,(\log K)^{1/2} \, K^{-G(\rho_{\min},\rho_{\max})},
$$
where $G$ is a positive rational function of $(\rho_{\min}, \rho_{\max})$, and $C'_{\rho_{\min},\rho_{\max}}>0$ is a constant free of $N$.
\end{theorem}
In particular, if we consider a sequence of $K_N$-banded correlation matrices $\Sigma_N \in \m B_{N, K_N}$ with $K_N \to \infty$ as $N \to \infty$, then under the conditions of Theorem \ref{thm:main}, $\lim_{N \to \infty} \alpha_{N, \delta} = 0$ for any fixed $\delta > 0$. Theorem \ref{thm:main}, being non-asymptotic in nature, additionally characterizes the rate of decay of $\alpha_{N, \delta}$. To contrast the conclusion of Theorem \ref{thm:main} with two closely related cases, consider first the case when $\theta \sim \m N(\mathbf{0}, \Sigma_N)$. For any $N$, the marginal distribution of $\theta_1$ is always $\m N(0, 1)$, and hence $\alpha_{N, \delta}$ does not depend on $N$. Similarly, if $\theta \sim \m N_{\m C}(\mathbf{0}, \Sigma_N)$ with $\Sigma_N$ a diagonal correlation matrix, then for any $N \ge 1$, the marginal distribution of $\theta_1$ is $\m N(0, 1)$ truncated to $(0, \infty)$ and $\alpha_{N, \delta}$ again does not depend on $N$. In particular, in both these cases, $\alpha_{N, \delta} \asymp \delta$ for $\delta$ small. However, when a combination of dependence and truncation is present, an additional $(\log K)^{1/2} \, K^{-G(\rho_{\min},\rho_{\max})}$ penalty is incurred. 

\begin{figure}[h]
  \begin{center}
    \includegraphics[width=0.50\textwidth]{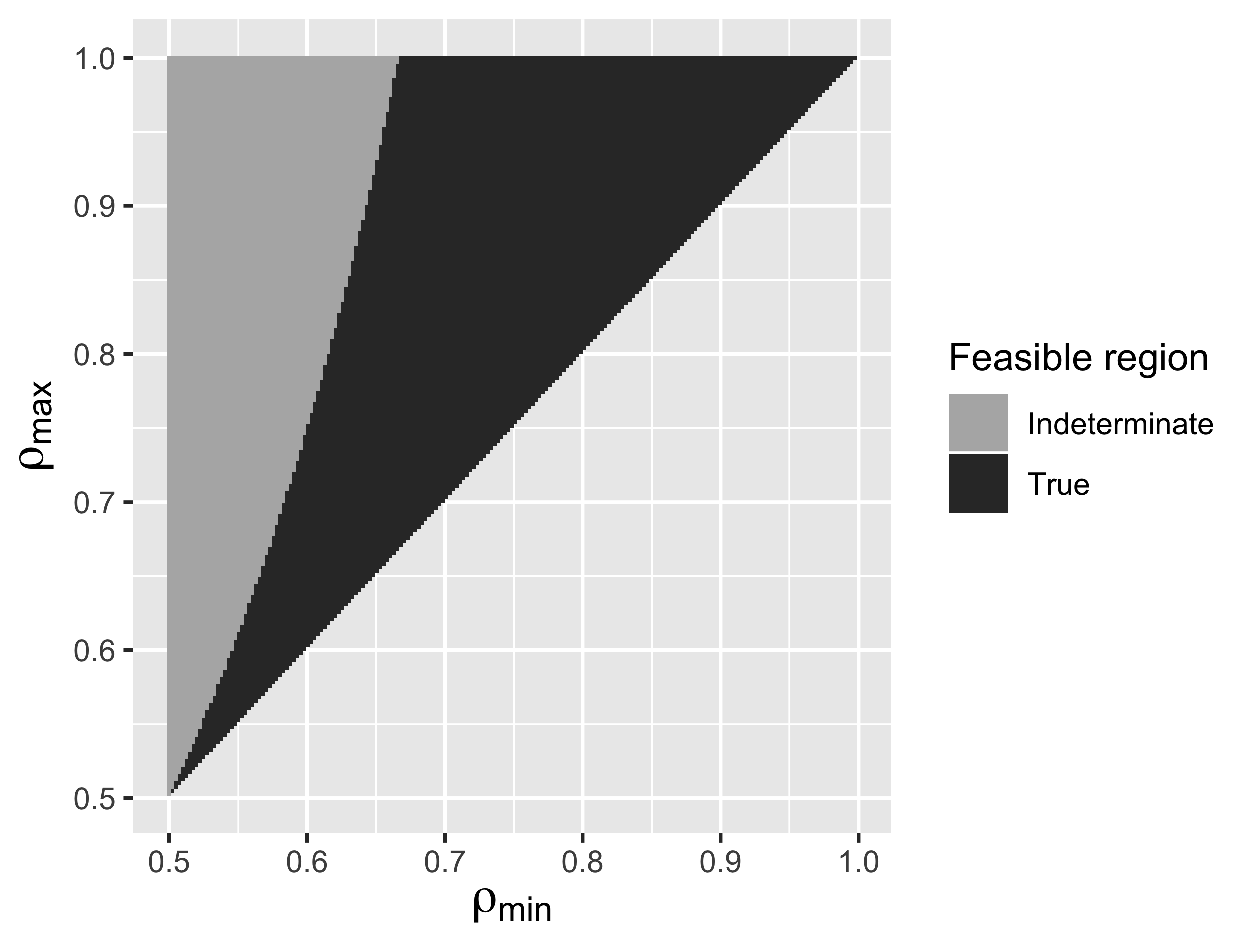}
  \end{center}
  \caption{The region shaded in black depicts $\m Q$ from the statement of Theorem \ref{thm:main}. }
  \label{fig:region}
\end{figure}

For the conclusion of Theorem \ref{thm:main} to hold, our current proof technique requires $(\rho_{\min}, \rho_{\max})$ to lie in the region $\m Q$, which is pictorially represented by the black shaded region in Figure \ref{fig:region}. As a special case, if all the non-zero correlations are the same, i.e., $\rho_{\min} = \rho_{\max}$, then the condition simplifies to $\rho_{\min} > 0.5$. More generally, if we write $\rho_{\min} = \kappa \rho_{\max}$ for some $\kappa \in (0, 1]$, then the condition reduces to $\rho_{\min} \ge 1 - \kappa/2$.  

\begin{remark}\label{rem:bet}
For any fixed $N$, the marginal density of $\theta_1$ evaluated at the origin, $\widetilde{p}_{1,N}(0) = \lim_{\delta \to 0} \alpha_{N, \delta}/\delta$. Theorem \ref{thm:main} thus implies in particular that $\lim_{N\to \infty}\widetilde{p}_{1,N}(0) = 0$. Also, for any fixed $1 \le k \le N$, if we denote $\beta_{N, k, \delta} = \bbP(\theta_1 \le \delta, \ldots, \theta_k \le \delta)$, it is immediate that $\beta_{N, k, \delta} < \alpha_{N, \delta}$, and hence $\lim_{N \to \infty} \beta_{N, k, \delta} = 0$, meaning the probability of a corner region is vanishingly small for large $N$. 
\end{remark}

We now empirically illustrate the conclusion of the theorem by presenting the univariate marginal density $\widetilde{p}_{1,N}$ for different values of the dimension $N$ and the bandwidth $K$. The density calculations were performed using the {\bf R} package {\bf tmvtnorm}, which is based on the numerical approximation algorithm proposed in \cite{cartinhour1990one} and subsequent refinements in  \cite{genz1992numerical,genz1993comparison,genz2009computation}. 
\begin{figure}[h]
	\centering	
	\begin{multicols}{3}
  \includegraphics[scale = 0.046]{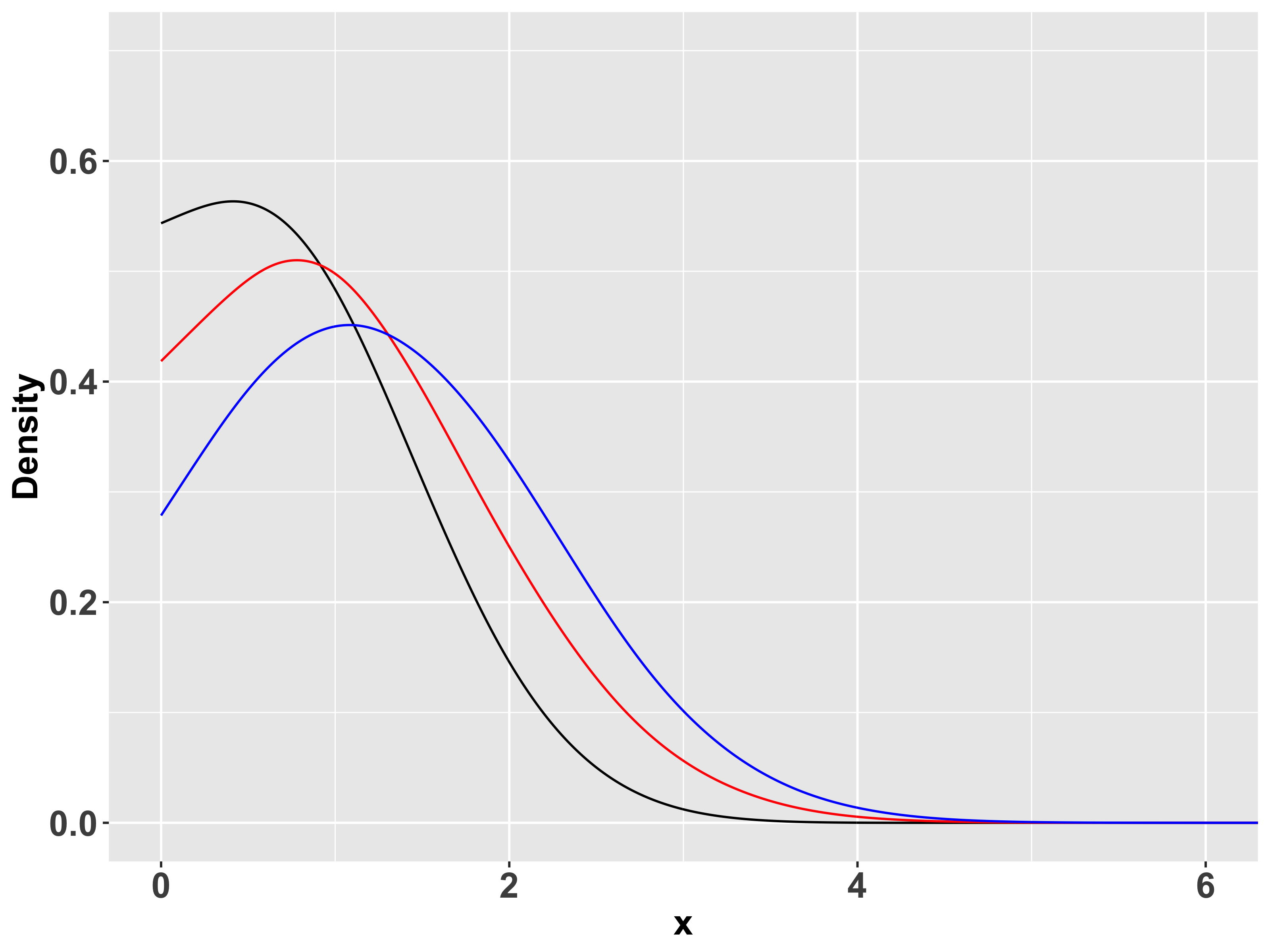}
   \includegraphics[scale = 0.046]{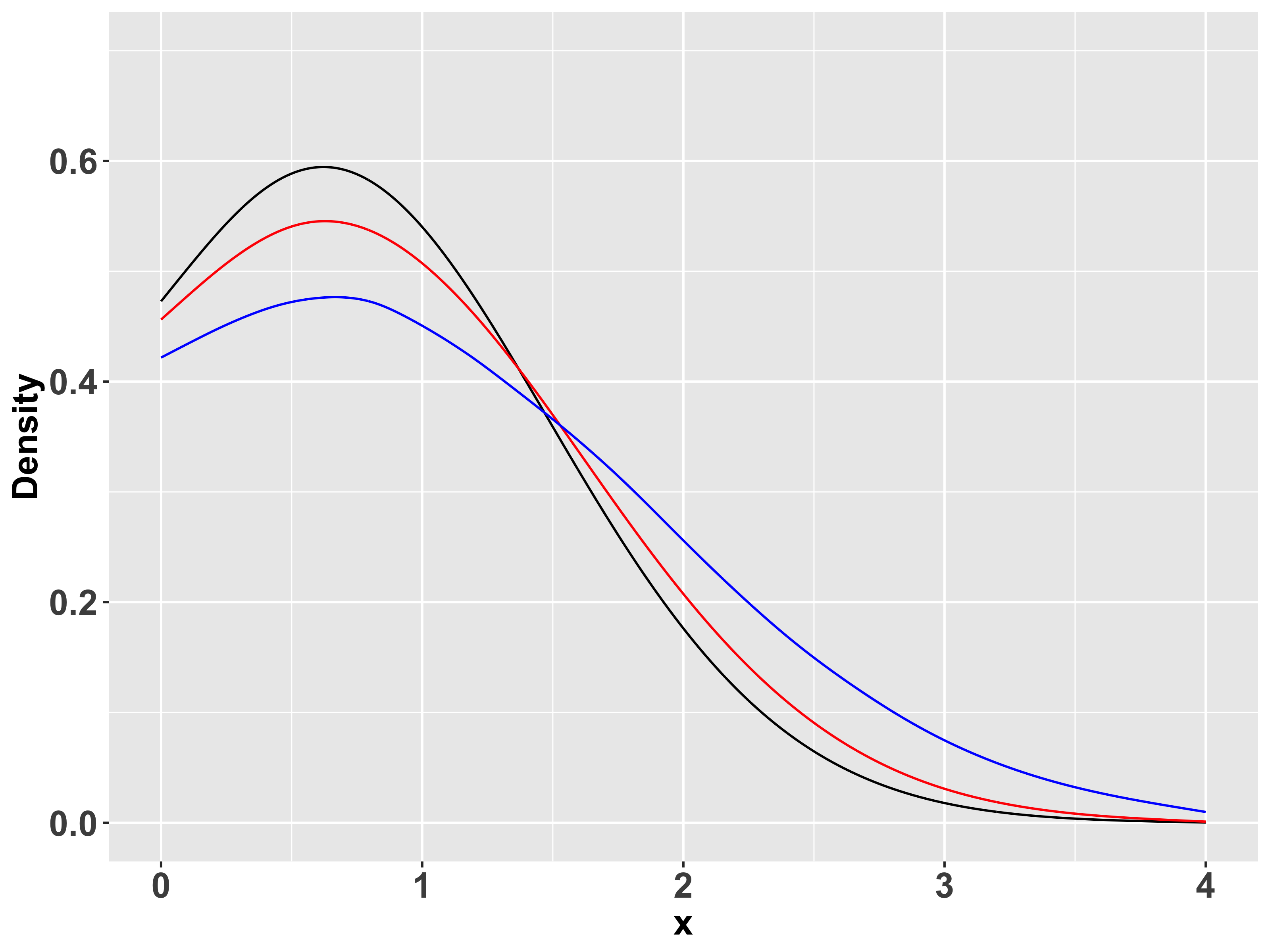}
   \includegraphics[scale = 0.046]{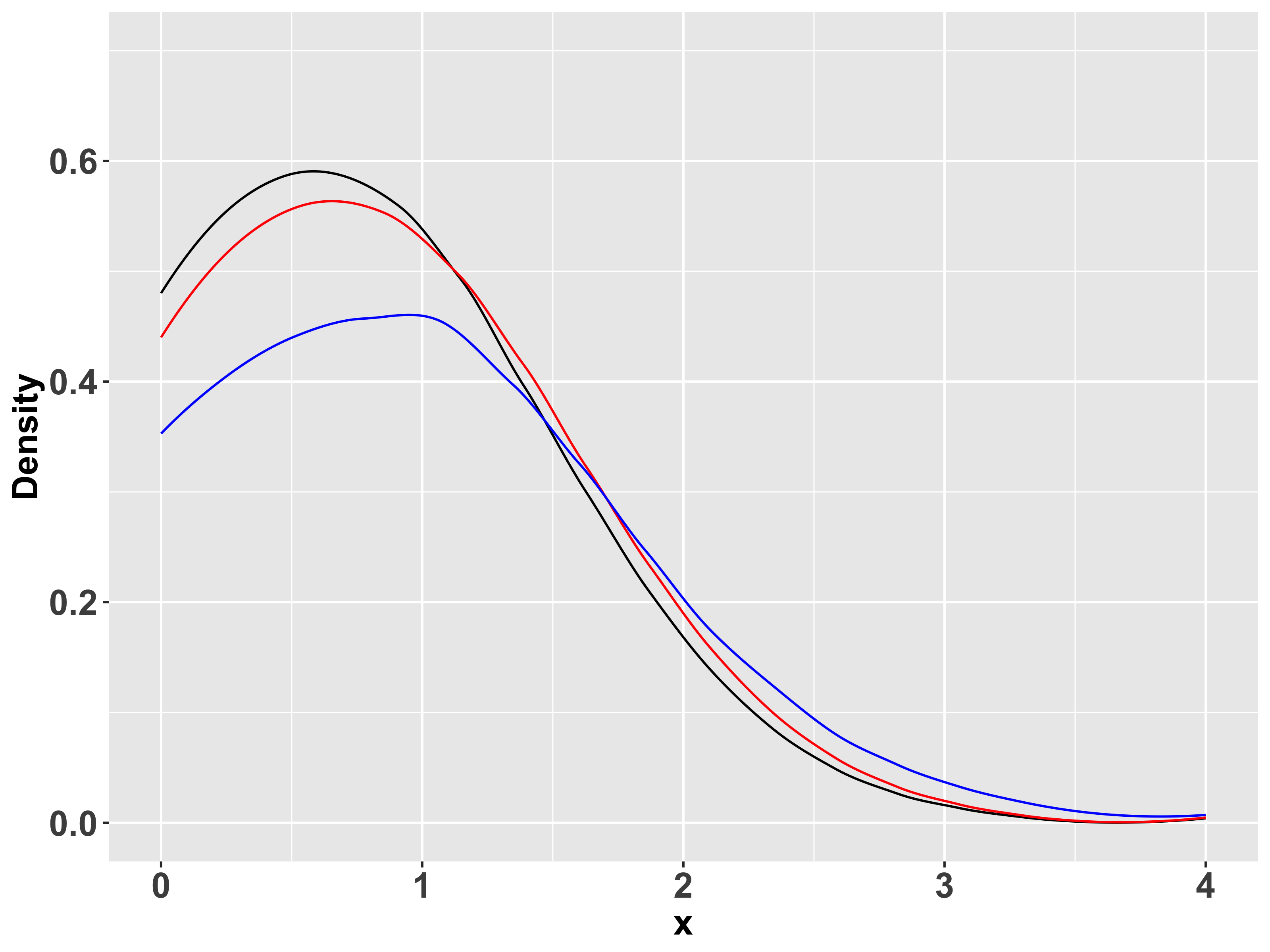}
\end{multicols}
\caption{\small {\em Left panel shows marginal density functions $\widetilde{p}_{1,N}$ for $K=2$ (black), $K=5$ (red) and $K=20$ (blue) with $N=100$. Middle panel shows $\widetilde{p}_{1,N}$ for $N=10$ (black), $N=50$ (red) and $N=100$ (blue) with $K=5$. Right panel shows $\widetilde{p}_{1,N}$ for $(K,N)=(5,25)$ (black), $(20,100)$ (red) and $(50,250)$ (blue).}}
	\label{fig:marg_den}
\end{figure} 
The left panel of Figure \ref{fig:marg_den} shows that for $N$ fixed at a moderately large value, the probability assigned to a small neighborhood of the origin decreases with increasing $K$. Also, the mode of the marginal density increasingly shifts away from zero. A similar effect is seen for a fixed $K$ and increasing $N$ in the middle panel and also for an increasing pair $(K,N)$ in the right panel, although the mass-shifting effect is somewhat weakened compared to the left panel. This behavior perfectly aligns with the main message of the theorem that the interplay between the truncation and the dependence brings forth the mass-shifting phenomenon. 


               
The proof of Theorem \ref{thm:main} is non-trivial; we provide the overall chain of arguments in the next subsection, deferring the proof of several auxiliary results to the supplemental document. The marginal probability of $(0, \delta)$ is a ratio of the probabilities of two rectangular regions under a $\m N(\mathbf{0}, \Sigma_N)$ distribution, with the denominator appearing due to the truncation. While there is a rich literature on estimating tail probabilities under correlated multivariate normals using multivariate extensions of the Mill's ratio \citep{savage1962mills,ruben1964asymptotic,sidak1968multivariate,steck1979lower,hashorva2003multivariate,lu2016note}, the existing bounds are more suited for numerical evaluation \citep{cartinhour1990one,genz1992numerical,genz2009computation} and pose analytic difficulties due to their complicated forms. Moreover, the current bounds lose their accuracy when the region boundary is close to the origin \citep{gasull2014approximating}, which is precisely our object of interest. Our argument instead relies on novel usage of Gaussian comparison inequalities such as the Slepian's inequality; see \cite{li2001gaussian,vershynin2016high} for book-level treatments. We additionally derive a generalization of Slepian's inequality in Lemma \ref{slepian_gl_main}, which might be of independent interest. As an important reduction step, we introduce a blocking idea to carefully approximate the banded scale matrix $\Sigma_N$ by a block tridiagonal matrix to simplify the analysis.

\subsection{Proof of Theorem \ref{thm:main}}\label{sketch}
By definition, 
\begin{align}\label{prob_ratio}
\alpha_{N,\delta} =\bbP(\theta_1 \le \delta) = \frac{\bbP(0\le Z_1 \le \delta, Z_2 \ge 0, \ldots, Z_N \ge0)}{\bbP(Z_1 \ge 0, Z_2\ge 0, \ldots, Z_N \ge0)}, 
\end{align}
where $Z\sim \m N(\mathbf{0}, \Sigma_N)$. We now proceed to separately bound the numerator and denominator in the above display. 

We first consider the denominator in equation \eqref{prob_ratio}, and use Slepian's lemma to bound it from below. It follows from Slepian's inequality, see comment after Lemma \ref{slepian} in the supplemental document, that if $X, Y$ are centered $d$-dimensional Gaussian random variables with $\bbE(X_i^2) = \bbE(Y_i^2)$ for all $i$, and $\bbE(X_i X_j) \le \bbE(Y_i Y_j)$ for all $i \ne j$, then 
\begin{align}\label{eq:Slep}
\bbP(X_1 \ge 0, \ldots, X_d \ge 0) \le \bbP(Y_1 \ge 0, \ldots, Y_d \ge 0). 
\end{align}

\begin{figure}[h!]
\centering
	\begin{subfigure}{.5\textwidth}
		\centering
		\includegraphics[width=\textwidth]{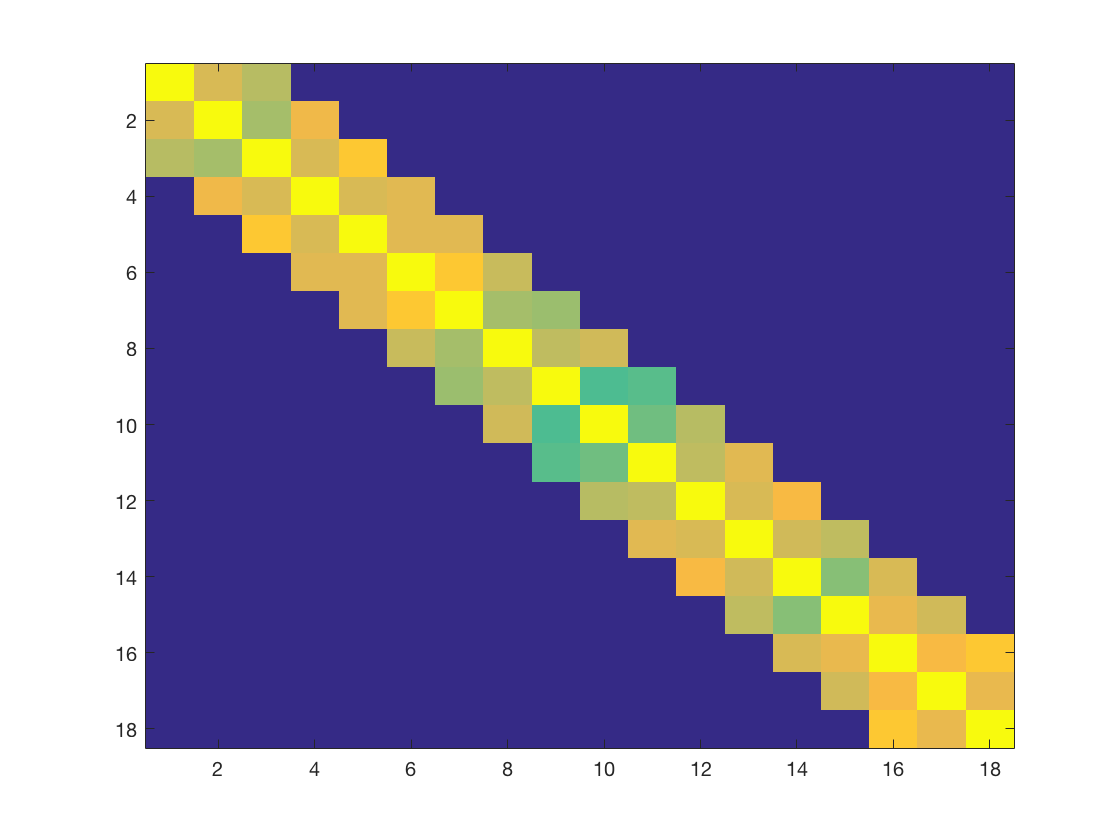}
	\end{subfigure}%
	\begin{subfigure}{.5\textwidth}
		\centering
		\includegraphics[width=\textwidth]{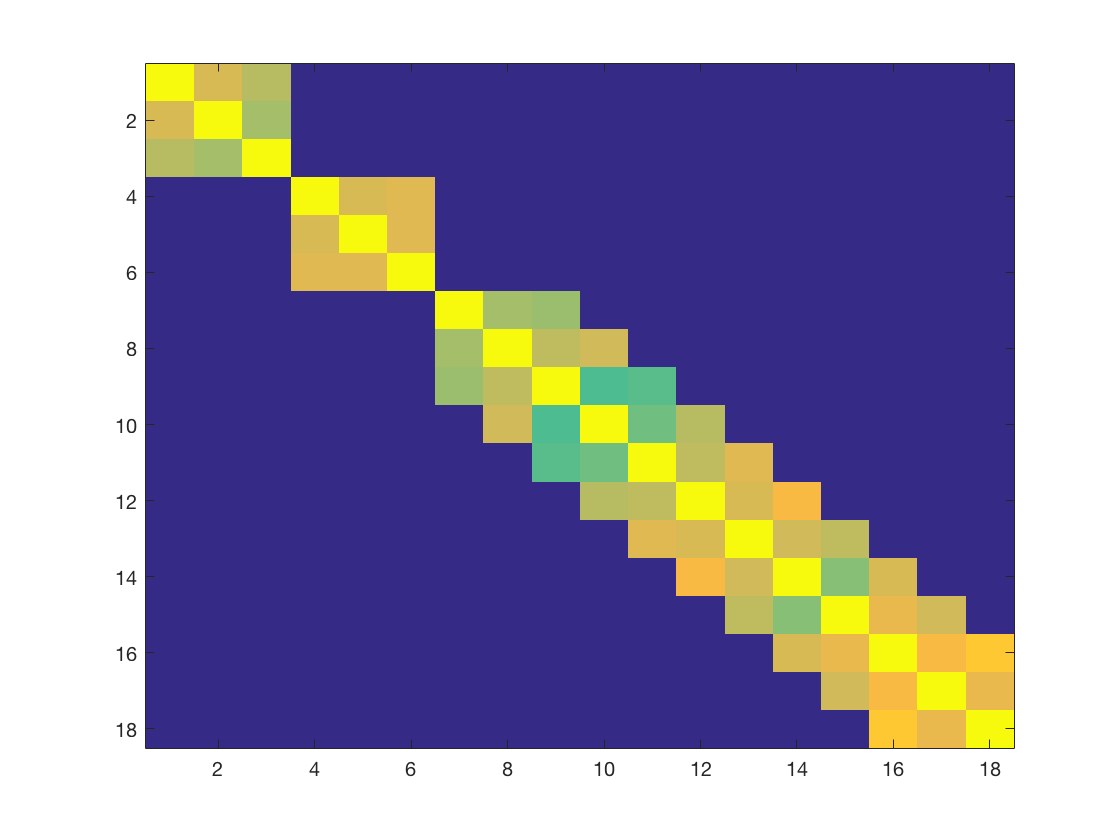}
	\end{subfigure}
	\caption{{\em Left panel: example of $\Sigma_N$ with $N = 18, K = 3$. Right panel: the corresponding block approximation $\wt{\Sigma}_N$.}}
	\label{fig-blk}
\end{figure}

The Slepian's inequality is a prominent example of a Gaussian comparison inequality originally developed to bound the supremum of  Gaussian processes. To apply Slepian's inequality to the present context, we construct another $N$-dimensional centered Gaussian random vector $S \sim \m N_N(\mathbf{0}, \wt{\Sigma}_N)$ such that (i) $S_{[1\,:\,K]} \overset{d}{=} Z_{[1\,:\,K]}$, $S_{[(K+1)\,:\,2K]}\overset{d}{=} Z_{[(K+1)\,:\,2K]}$ and $S_{[(2K+1)\,:\,N]} \overset{d}{=} Z_{[(2K+1)\,:\,N]}$, and (ii) the sub-vectors $S_{[1\,:\,K]}$, $S_{[(K+1)\,:\,2K]}$ and $S_{[(2K+1)\,:\,N]}$ are mutually independent. The correlation matrix $\wt{\Sigma}_N$ of $S$ clearly satisfies $(\Sigma_N)_{ij} \ge (\wt{\Sigma}_N)_{ij}$ for all $i \ne j$ by construction. Figure \ref{fig-blk} pictorially depicts this block approximation in an example with $N = 18$ and $K = 3$. Applying Slepian's inequality, we then have, 
\begin{align}\label{low_bd1_main}
&\bbP(Z_1\ge 0,\ldots, Z_N\ge 0) \ge \bbP(S_1\ge 0,\ldots, S_N\ge 0) \nonumber \\
&~~~=\bbP(S_{[1\,:\,K]}\ge \mathbf{0})\,\bbP(S_{[(K+1)\,:\,2K]}\ge \mathbf{0})\,\bbP(S_{[(2K+1)\,:\,N]}\ge \mathbf{0}) \nonumber \\
&~~~=\bbP(Z_{[1\,:\,K]}\ge \mathbf{0})\,\bbP(Z_{[(K+1)\,:\,2K]}\ge \mathbf{0})\,\bbP(Z_{[(2K+1)\,:\,N]}\ge \mathbf{0}).
\end{align} 

Next, we consider the numerator in equation \eqref{prob_ratio}. We have, 
\begin{align}\label{upp_bd1_main}
&\bbP(0\le Z _1 \le \delta, Z_2\ge 0, \ldots, Z_N \ge0) \nonumber \\
&~~~~~~~~\le \bbP\big(0\le Z_1 \le \delta, Z_{[2\,:\,K]}\ge \mathbf{0}, \, Z_{[(K+1)\,:\,2K]}\in \mathbb{R}^K,\,Z_{[(2K+1)\,:\,N]}\ge \mathbf{0}\big) \nonumber \\
&~~~~~~~~= \bbP\big(0 \le Z_1 \le \delta, Z_{[2\,:\,K]}\ge \mathbf{0}\big)\bbP\big(Z_{[(2K+1)\,:\,N]}\ge \mathbf{0}\big).
\end{align}
The last equality crucially uses $Z_{[1\,:\,K]}$ and $Z_{[(2K+1)\,:\,N]}$ are independent, which is a consequence of $\Sigma_N$ being $K$-banded. Taking the ratio of equations \eqref{low_bd1_main} and \eqref{upp_bd1_main}, the term $\bbP(Z_{[(2K+1)\,:\,N]}\ge \mathbf{0})$ cancels so that 
\begin{align}\label{eq:intermd}
\alpha_{N,\delta} \le \frac{ \bbP(0\le Z_1 \le \delta, Z_{[2\,:\,K]}\ge \mathbf{0})}{\bbP(Z_{[1\,:\,K]}\ge \mathbf{0})\,\bbP(Z_{[K+1\,:\,2K]}\ge \mathbf{0})} = R.
\end{align}
To bound the terms $\bbP(Z_{[1\,:\,K]}\ge \mathbf{0})$ and $\bbP(Z_{[K+1\,:\,2K]}\ge \mathbf{0})$ in the denominator of $R$, we resort to another round of Slepian's inequality. Let $Z'' \sim \m N(\mathbf{0}, \Sigma_K(\rho_{\min}))$, where recall that 
$\rho_{\min}$ is the minimum non-zero correlation in $\Sigma_N$. Also, recall from equation \eqref{eq:comp_symm} that $\Sigma_K(\rho_{\min})$ denotes the $K \times K$ compound-symmetry correlation matrix with all correlations equal to $\rho_{\min}$. By construction, for any $1 \le i \ne j \le K$, $\bbE (Z_i Z_j), \bbE(Z_{K+i} Z_{K+j}) \ge \rho_{\min} = \bbE(Z''_i Z''_j)$. Thus, applying Slepian's inequality as in equation \eqref{eq:Slep}, 
\be 
\bbP(Z_{[1\,:\,K]}\ge \mathbf{0})\,\bbP(Z_{[K+1\,:\,2K]}\ge \mathbf{0}) \ge \{\bbP(Z'' \ge \mathbf{0})\}^2. 
\ee

The numerator of equation \eqref{eq:intermd} cannot be directly tackled by Slepian's inequality, and we prove the following comparison inequality in the supplemental document. 
\begin{lemma}\label{slepian_gl_main}
(Generalized Slepian's inequality) Let $X, Y$ be centered $d$-dimensional Gaussian vectors with $\bbE X_i^2 = \bbE Y_i^2$ for all $i$ and $\bbE (X_iX_j) \le \bbE (Y_iY_j)$ for all $i \ne j$. Then for any $0\le \ell_1 < u_1$ and $u_2, \ldots u_d \in \mb R$,  we have 
\be
\bbP\big(\ell_1 \le X_1\le u_1 , X_2 \ge u_2, \ldots, X_d \ge u_d\big) \le \bbP\big(\ell_1 \le Y_1 \le u_1, Y_2 \ge u_2, \ldots, Y_d \ge u_d\big).
\ee
\end{lemma}
Define a random variable $Z'\sim \m N(\mathbf{0}, \Sigma_{K}(\rho_{\max}))$ and use Lemma \ref{slepian_gl_main} to conclude that $\bbP(0\le Z_1 \le \delta, Z_{[2\,:\,K]}\ge \mathbf{0}) \le \bbP(0\le Z'_1 \le \delta, Z'_{[2\,:\,K]}\ge \mathbf{0})$. 

Substituting these bounds in equation \eqref{eq:intermd}, we obtain 
\begin{align}\label{eq:intermd1}
R \le R' = \frac{\bbP(0\le Z'_1 \le \delta, Z'_2\ge 0, \ldots, Z'_K\ge 0)}{\{\bbP(Z''_1\ge 0, \ldots, Z''_K\ge 0)\}^2}. 
\end{align}
The primary reduction achieved by bounding $R'$ by $R''$ is that we only need to estimate Gaussian probabilities under a compound-symmetry covariance structure. We prove the following inequalities in the supplemental document that provide these estimates. 
\begin{lemma}\label{upper_lower_small}
Let $X \sim \m N(\mathbf{0}, \Sigma_d(\rho))$ with $\rho \in (0, 1)$. Fix $\delta>0$. Define $\bar{\rho}=(1-\rho)/\rho$. Then, 
\begin{align*}
&\bbP(0\le X_1<\delta,X_2\ge 0, \ldots, X_d \ge 0) \nonumber\\
&~~~~~~~~~~~~\le \delta\, \{2(1-\alpha)\bar{\rho}\log (d-1)\}^{-1/2}\, (d-1)^{-(1-\alpha)/\rho} + \exp(-d^\alpha),
\end{align*}
for any $\alpha\in(0,1)$. Also, 
$$
\bbP(X_1 \ge 0, \ldots, X_d \ge 0) \ge \frac{(2\,\bar{\rho}\log d)^{1/2}}{2\,\bar{\rho}\log d+1}\, d^{-\bar{\rho}}.
$$
\end{lemma}
A key aspect of the compound-symmetry structure that we exploit is for $X\sim \m N (\mathbf{0}, \Sigma_d(\rho))$ with $\rho \in (0,1)$, 
we can represent  $X_i \overset{d}= \rho^{1/2}\, w + (1- \rho)^{1/2} \, W_i$, where $w, W_i$'s are independent $\m N(0, 1)$ variables. 

Using Lemma \ref{upper_lower_small}, we can bound, for any $\alpha \in (0, 1)$, 
\begin{align}\label{eq:bdR2} 
R' &\le \frac{ \delta\,\{2\,\bar{\rho}_{\max}\,(1-\alpha)\log (K-1)\}^{-1/2}\, (K-1)^{-(1-\alpha)/\rho_{\max}} + \exp\{-(K-1)^\alpha\}}{2\, \bar{\rho}_{\min}\log K(2\,\bar{\rho}_{\min} \log K+1)^{-2}K^{-2\,\bar{\rho}_{\min}}} \nonumber \\
&\le 2^{(1-\alpha)/\rho_{\max}}\,\delta\,\frac{(2\,\bar{\rho}_{\min}\log K+1)^2}{2\,\bar{\rho}_{\min}\,\log K\{2\,\bar{\rho}_{\max}\,(1-\alpha)\log (K-1)\}^{1/2}}\, K^{-\{(1-\alpha)/\rho_{\max} - 2\,\bar{\rho}_{\min}\}} \nonumber \\
 &~~~+ \exp\{-(K-1)^\alpha\} K^{2\,\bar{\rho}_{\min}} (2\,\bar{\rho}_{\min} \log K+1)^2/ (2\,\bar{\rho}_{\min} \log K)\nonumber \\
 &\le C \,\delta\,(\log K)^{1/2}\,K^{-\{(1-\alpha)/\rho_{\max} - 2\,\bar{\rho}_{\min}\}} + 4 \,\bar{\rho}_{\min} \exp\{-(K-1)^\alpha\} K^{2\,\bar{\rho}_{\min}}\log K,
\end{align}
with $C = 5\bar{\rho}_{\min}/\{(1-\alpha)\bar{\rho}_{\max}\}^{1/2}$. Since $(\rho_{\min}, \rho_{\max}) \in \m Q$, we have $\rho_{\min}/\{2 (1 - \rho_{\min})\} \ge \rho_{\max}$, or equivalently, $2 \bar{\rho}_{\min} < 1/\rho_{\max}$. Thus, we can always find $\alpha > 0$ such that $(1-\alpha)/\rho_{\max} - 2\,\bar{\rho}_{\min} > 0$. Fix such an $\alpha$, and substitute in equation \eqref{eq:bdR2}. The proof is now completed by choosing $K_0$ large enough so that for any $K > K_0$, the second term in the last line of \eqref{eq:bdR2} is smaller than the first; this is possible since the second term decreases exponentially while the first does so polynomially in $K$.

\section{Connections with Bayesian constrained inference}\label{sec:band}
In this section, we connect the theoretical findings in the previous section to posterior inference in Bayesian constrained regression models. We work under the setup of a usual Gaussian regression model, 
\begin{align}\label{eq:model}
y_i = f(x_i) + \epsilon_i,  \quad  \epsilon_i  \sim \m N(0, \sigma^2) \quad (i = 1,\dots, n),
\end{align}
where we assume $x_i \in [0, 1]$ for simplicity. We are interested in the situation when the regression function $f$ is constrained to lie in some space $\m C_f$ which is a subset of the space of all continuous functions on $[0, 1]$, determined by linear restrictions on $f$ and possibly its higher-order derivatives. Common examples include bounded, monotone, convex, and concave functions. 

As discussed in the introduction, a general approach is to expand $f$ in some basis $\{\phi_j\}$ as $f(\cdot) = \sum_{j=1}^N \theta_j \phi_j(\cdot)$ so that the restrictions on $f$ can be posed as linear restrictions on the vector of basis coefficients $\theta \in \mb R^N$, with the parameter space $\m C$ for $\theta$ of the form $\m C = \{\theta \in \mb R^N\,:\, A \theta \ge b\}$. For example, when $\m C_f$ corresponds to monotone increasing functions, the set $\m C$ is of the form $\{\theta_1 \le \theta_2 \ldots \le \theta_N\}$ under the Bernstein polynomial basis \citep{curtis2009variable} and $[0, \infty)^N$ under the integrated triangular basis of \cite{maatouk2017gaussian}. For sake of concreteness, we shall henceforth work with $\m C = [0, \infty)^N$. 
Under such a basis representation, the model \eqref{eq:model} can be expressed as
\begin{align}\label{model_tmvn}
 Y &= \Phi \theta + \epsilon, \quad \epsilon \sim \m N({\bf 0}, \mr I_n),  \quad \theta \in \m C,%
\end{align}
where $Y = (y_1, \ldots, y_n)^\T$ and $\Phi = \{\phi_j(x_i)\}_{ij}$ is an $n\times N$ basis matrix. 

The truncated normal prior $\theta \sim \m N_{\m C}(\mathbf{0}, \Omega_N)$ is conjugate, with the posterior $\theta \mid Y \sim \m N_{\m C}(\mu_N, \Sigma_N)$, with 
$$
\mu_N =  \Sigma_N\Phi^\T Y, \quad \Sigma_N= (\Omega_N^{-1}+\Phi^\T\Phi)^{-1}. 
$$
To motivate their prior choice, \cite{maatouk2017gaussian} begin with an unconstrained mean-zero Gaussian process prior on $f$, $f \sim \textsc{gp}(0, K)$, with covariance kernel $K$. Since their basis coefficients correspond to evaluation of the function and its derivatives at the grid points; see Appendix \ref{sec:app_posterior} for details; this induces a multivariate zero-mean Gaussian prior $\m N(\mathbf{0}, \Sigma_N)$ on $\theta$ provided the covariance kernel $K$ of the parent Gaussian process is sufficiently smooth. Having obtained this unconstrained Gaussian prior on $\theta$, \cite{maatouk2017gaussian} multiply it with the indicator function $\ind_{\m C}(\theta)$ of the truncation region to obtain the truncated normal prior. 

We are now in a position to connect the posterior bias in Figures \ref{fig-cgp} and \ref{fig-cgp1} to the mass-shifting phenomenon characterized in the previous section. Since the posterior $\theta \mid Y \sim \m N_{\m C}(\mu_N, \Sigma_N)$, a draw from the posterior can be represented as 
$$
\theta = \mu_N + \theta_c, \quad \theta_c \sim \m N_{\m C}(0, \Sigma_N). 
$$
Consider an extreme scenario when the true function is entirely flat. In this case, the optimal parameter value $\theta_0 = \mathbf{0}_N$ and under mild assumptions, $\mu_N$ is concentrated near the origin with high probability under the true data distribution; see \S \ref{sec:mean} of the supplementary material. The mass shifting phenomenon pushes $\theta_c$ away from the origin, resulting in the bias. On the other hand, when the true function is strictly monotone as in the left panel of Figure \ref{fig-cgp}, all the entries of $\mu_N$ are bounded away from zero, which masks the effect of the shift in $\theta_c$. 

In strict technical terms, our theory is not directly applicable to $\theta_c$ since the scale matrix $\Sigma_N$ is a dense matrix in general. However, we show below that $\Sigma_N$ is approximately banded under mild conditions. Figure \ref{fig:band} shows image plots of $\Sigma_N$ for three choices of $N$ using the basis of \cite{maatouk2017gaussian} and sample size $n = 500$. In all cases, $\Sigma_N$ is seen to have a near-banded structure.   

\begin{figure}[h]
	\centering	
\begin{multicols}{3}
  \includegraphics[scale = 0.048]{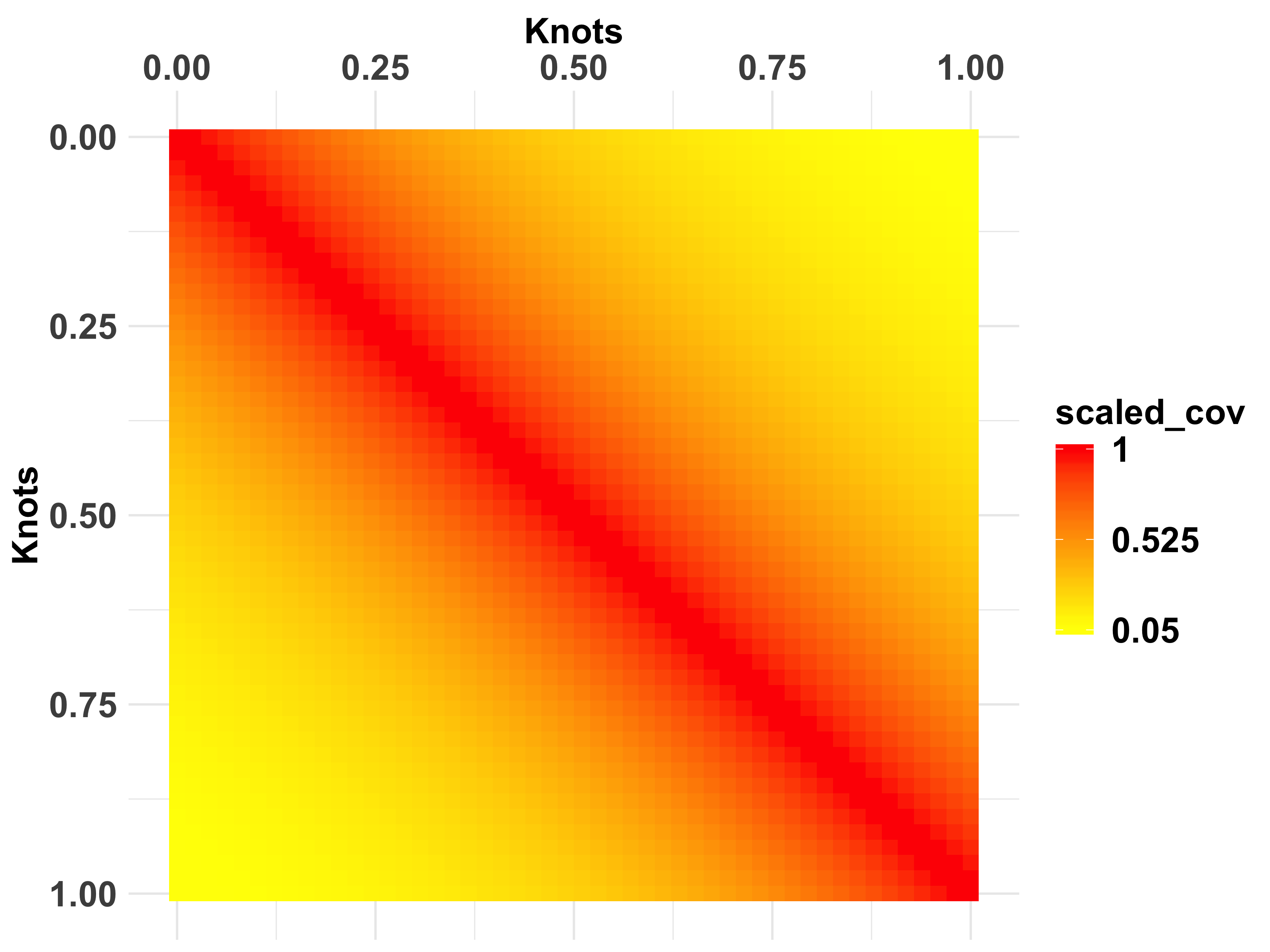}
   \includegraphics[scale = 0.048]{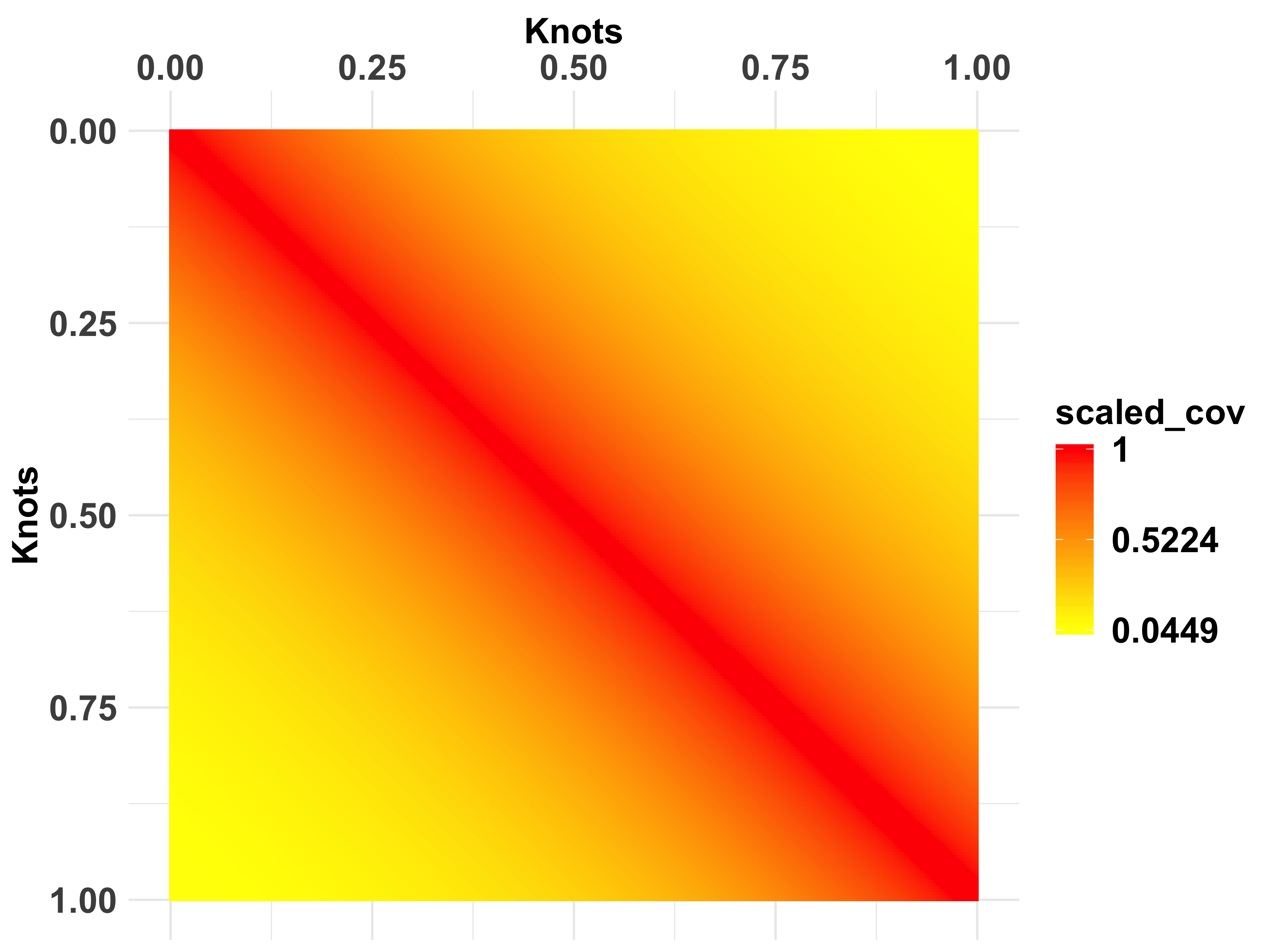}
   \includegraphics[scale = 0.048]{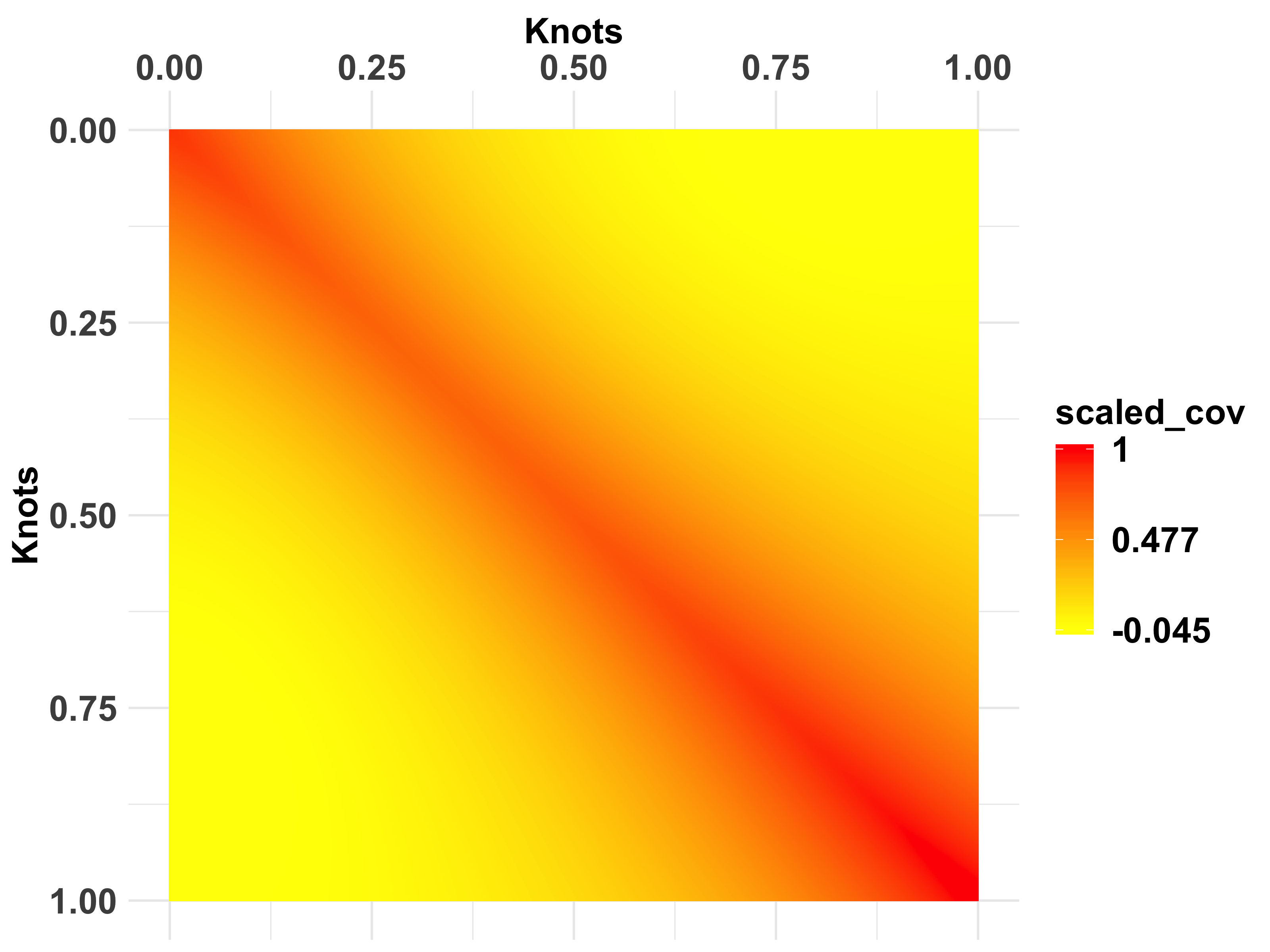}
\end{multicols}
        \caption{Scaled posterior scale matrix $\widetilde{\Sigma}_N$ of dimension $N = 50$ (left), $N = 250$ (middle) and $N=500$ (right).}
	\label{fig:band}
\end{figure}

We make this empirical observation concrete below. We first state the assumptions on the basis matrix $\Phi$ and prior covariance matrix $\Omega_N$ that allows the construction of a strictly banded matrix approximation.
\begin{assumption} \label{ass:basis}
We assume the basis matrix $\Phi$ is such that the matrix $\Phi^\T\,\Phi$ is $q$-banded for some $2\le q \le N$; also there exists constants  $0< C_1<C_2 <\infty$ such that 
\begin{align*}
C_1\,(n/N)\,I_N \le \Phi^\T\Phi \le C_2\,(n/N)\,I_N.
\end{align*}   
\end{assumption}
One example of a basis satisfying Assumption \ref{ass:basis} is a B-Spline of fixed order $q$ denoted as $B_{N,q}(x)$ with $N=J+q$ over quasi-uniform knot points of number $J>0$; see, for example, \cite{yoo2016supremum}. 

Regarding the prior covariance matrix, we first define a uniform class of symmetric positive definite well-conditioned matrices \citep{bickel2008regularized} as
\begin{align} \label{wellcond}
\m M(\lambda_0, \alpha, k) = \Big\{ \Omega_N\,:\,&\max_j \sum_i \{|\Sigma_{ij}|\,:\,|i-j|>k\} \le C\,k^{-\alpha}\, \mbox{for all} \,k>0,  \\ \nonumber
~~~~~~~~~~~~~~~~~~&\mbox{and} \, 0 <\lambda_0 \le \lambda_{\min} (\Omega_N) \le \lambda_{\max} (\Omega_N) \le 1/\lambda_0 \Big\}.
\end{align}

\begin{assumption} \label{ass:band}
We assume the prior covariance matrix $ \Omega_N \in \m M(\lambda_0, \alpha, k)$ defined in \eqref{wellcond}.
\end{assumption}
Assumption \ref{ass:band} ensures the covariance matrix is ``approximately bandable", which is common in covariance matrix estimation with thresholding techniques \citep{bickel2008covariance,bickel2008regularized}. 

Given above Assumptions, we are now ready to give the approximation result of posterior scale matrix $\Omega^{-1}$ to a banded symmetric positive definite matrix.
\begin{proposition}\label{prop:app}
For the posterior scale matrix $\Sigma_N = (\Omega_N^{-1}+\Phi^\T\Phi)^{-1}$ with $\Phi$ satisfying Assumption \ref{ass:basis} and $\Omega_N$ satisfying Assumption \ref{ass:band}, for sufficiently small $0<\epsilon<1/\lambda_0$ there exists $r \succsim \log(1/\epsilon)$, and for sufficiently large $n_0$ we can always find a $\max(n^2_0r,n_0q)$-banded, symmetric and positive definite matrix $\Sigma'_N$ such that
\begin{align}\label{dist}
\|\Sigma_N - \Sigma'_N\| \precsim \delta_{\epsilon,\kappa},
\end{align}
where $\delta_{\epsilon,\kappa} = (\epsilon+\kappa^{n_0+1})\max\{(N/n),(N/n)^2\}$ and $0< \kappa <1$ is a fixed constant. 
\end{proposition}

Proposition \ref{prop:app} states under mild conditions we can always construct a banded positive definite matrix that approximates $\Sigma_N$ in operator norm.  Applying the result in Theorem \ref{thm:main} to a truncated normal distribution with the banded approximation of the posterior scale matrix, the marginal density would present a mass-shifting behavior. If we control the band width $K$ such that the approximation is close enough to the posterior scale matrix, the marginal posterior distribution would be expected to behave similarly and shift its probability mass away from the origin and the probability mass over the ``corner region" will decrease to zero as the dimension $N$ goes to infinity. This helps explain the bias occurred in the posterior mean over the flat area shown in the Figure \ref{fig-cgp}.

\section{A de-biasing remedy based on a shrinkage prior}\label{sec:meth}

As concluded in previous sections, we view the mass-shifting behavior of the posterior marginals causing the bias in posterior estimation for flat functions. In this section, we will provide empirical evidences that a simple modification to the truncated normal prior can alleviate the issues related to such mass-shifting phenomenon. Among remedies proposed in the literature, \cite{curtis2009variable} proposed independent shrinkage priors on the parameter vector $\{u_k\}$ given by a mixture of a point-mass at zero and a univariate normal distribution truncated to the positive real line as an alternative to the truncated normal prior,
$$
u_k \sim (1 - \pi) \delta_0 + \pi \m N_+(\mu, \sigma^2). 
$$
Similar mixture priors were also previously used by \cite{neelon2004bayesian} and \cite{dunson2005bayesian}. The mass at zero allows positive prior probability to functions having exactly flat regions. Although possible in principle, introduction of such point-masses while retaining the dependence structure between the coefficients becomes somewhat cumbersome in addition to being computationally burdensome. With such motivation and the additional consideration that in most real scenarios a function is approximately flat in certain regions, we propose a shrinkage procedure as a remedy to replace the coefficients $\theta \in \m C$ by $\xi = (\xi_1, \ldots, \xi_N)^\T$, where
\begin{align}\label{shrink_eq}
\xi_j = \tau \, \lambda_j \, \theta_j,  \quad (j = 1, \ldots, N)
\end{align}
The parameter $\tau$ provides global shrinkage towards the origin while the $\lambda_j$s provide coefficient-specific deviation. We consider default \citep{carvalho2010horseshoe} half-Cauchy priors $\m C_{+}(0,1)$ on $\tau$ and the $\lambda_j$s independently. The $\mc C_{+}(0,1)$ distribution has a density proportional to $(1+t^2)^{-1} \mathbbm{1}_{(0, \infty)}(t)$. We continue to use a dependent truncated normal prior $\theta \sim \m N_{\m C}(\mathbf{0}_N, \Sigma_N)$ which in turn induces dependence among the $\xi_j$s. Our prior on $\xi$ can thus be considered as a dependent extension of the global-local shrinkage priors \citep{carvalho2010horseshoe} widely used in the high-dimensional regression context. Figure \ref{prior-draw} in \S \ref{app:prior-draw} of the supplementary material shows prior draws for the first and third components of both $\theta$ and $\xi$, based on which the marginal distribution of the $\xi_j$s is clearly seen to place more mass near the origin while retaining heavy tails.

We provide an illustration of the proposed shrinkage procedure in the context of estimating monotone functions as described in \eqref{eq:mc}. The procedure can be readily adapted to include various other constraints. Replacing $\theta$ by $\xi$ in (M) in \eqref{eq:mc}, 
we can write \eqref{eq:model} in vector notation as 
\begin{eqnarray}\label{model}
Y = \xi_0 \mathbf{1}_n + \tau \, \Psi \Lambda \theta + \varepsilon, \quad \varepsilon \sim \mc N_n(\mathbf{0}_n, \sigma^2 \mr I_n).
\end{eqnarray}
Here, $\Psi$ is an $n \times N$ basis matrix with $i^{th}$ row $\Psi_i^\T$ where 
$\Psi_{ij} = \psi_{j-1}(x_i)$ for $j=1, \ldots, N$ and the basis functions $\psi_j$ are as in \eqref{eq:mc}. 
Also, 
$Y = (y_1, \ldots, y_n)^\T$, $\Lambda = \mbox{diag}(\lambda_1, \ldots, \lambda_{N})$ and $\varepsilon = (\epsilon_1, \ldots, \epsilon_n)^\T$.

The model is parametrized by $\xi_0 \in \mb R$, $\theta = (\theta_1, \ldots, \theta_N)^\T \in \mc C$, $\lambda = (\lambda_1, \ldots, \lambda_N)^\T \in \mc C$, $\sigma \in \mb R^+$ and $\tau \in \mb R^+$. We place a flat prior $\pi(\xi_0) \propto 1$ on $\xi_0$.  We place a truncated normal prior $\m N_{\m C}(\mathbf{0}_N, \Sigma_N)$ on $\theta$ independently of $\xi_0$, $\tau$ and $\lambda$ with 
$$
\Sigma_N = (\Sigma_{jj'}), \quad \Sigma_{jj'} = k(u_j - u_{j'}), \quad u_j = j/(N-1), \quad (j = 0, 1, \ldots, N-1)  
$$
and $k(\cdot)$ is the stationary Mat{\'e}rn kernel with smoothness parameter $\nu > 0$ and length-scale parameter $\ell > 0$. To complete the prior specification, we place improper prior $\pi (\sigma^2) \propto 1/\sigma^2$ on $\sigma^2$ and compactly supported priors $\nu \sim \mc U(0.5,1)$ and $\ell \sim \mc U(0.1,1)$ on $\nu$ and $\ell$. 
We develop a data-augmentation Gibbs sampler which combined with the 
embedding technique of \cite{ray2019efficient} results in an efficient MCMC algorithm to sample from the joint posterior of $(\xi_0, \theta, \lambda, \sigma^2, \tau^2, \nu, \ell)$; the details are deferred to \S \ref{post-sam} of the supplementary material. 
	


We conduct a small-scale simulation study to illustrate the efficacy of the proposed shrinkage procedure.  We consider model \eqref{eq:model} with true $\sigma = 0.5$ and two different choices of the true $f$, namely, 
$$
f_1(x) = (5x - 3)^3 \,\ind_{[0.6, 1]}(x), \quad f_2(x) = \sqrt{2} \, \sum_{l=1}^{100} l^{-1.7} \, \sin(l) \, \cos(\pi (l-0.5) (1-x))
$$
for $x \in [0, 1]$. The function $f_1$, which is non-decreasing and flat between $0$ and $0.6$, was used as the motivating example in the introduction. The function $f_2$ is also approximately flat between $0.7$ and $1$, although it is not strictly non-decreasing in this region, which allows us to evaluate the performance under slight model misspecification. 

To showcase the improvement due to the shrinkage, we consider a cascading sequence of priors beginning with only a truncated normal prior and gradually adding more structure to eventually arrive at the proposed shrinkage prior. Specifically, the variants considered are \\
{ \em No shrinkage and fixed hyperparameters:} Here, we set $\Lambda = \mr I_N$ and $\tau = 1$ in \eqref{model}, and also fix $\nu$ and $\ell$, so that we have a truncated normal prior on the coefficients. This was implemented as part of the motivating examples in the introduction. We fix $\nu = 0.75$ and $\ell$ so that the correlation $k(1)$ between the maximum separated points in the covariate domain equals $0.05$. \\
{\em No shrinkage with hyperparameter updates:} The only difference from the previous case is that $\nu$ and $\ell$ are both assigned priors described previously and updated within the MCMC algorithm. \\
{\em Global shrinkage:} We continue with $\Lambda = \mr I_N$ and place a half-Cauchy prior on the global shrinkage parameter $\tau$. The hyperparameters $\nu$ and $\ell$ are updated. \\
{\em Global-local shrinkage:} This is the proposed procedure where the $\lambda_j$s are also assigned half-Cauchy priors and the hyperparameters are updated. 

We generate $500$ pairs of response and covariates and randomly divide the data into $300$ training samples and $200$ test samples. For all of the variants above, we set the number of knots $N = 150$. 
We provide plots of the function fit along with pointwise $95 \%$ credible intervals in Figures \ref{shrink-f1} and \ref{shrink-f2} respectively, and also report the mean squared prediction error (\textsc{mspe}) at the bottom of the sub-plots. As expected, only using the truncated normal prior leads to a large bias in the flat region. Adding some global structure to the truncated normal prior, for instance, updating the \textsc{gp} hyperparamaters and adding a global shrinkage term improves estimation around the flat region, which however still lacks the flexibility to transition from the flat region to the strictly increasing region. The global-local shrinkage performs the best, both visually and also in terms of \textsc{mspe}. 

Additionally, the shrinkage procedure performed at least as good as bsar, a very recent state-of-the-art method, developed by \cite{lenk2017bayesian}, and implemented in the \textbf{{\fontfamily{qcr} \selectfont R}} package \textbf{{\fontfamily{qcr} \selectfont bsamGP}}. For the out-of-sample prediction performance of bsar, refer to Figure \ref{bsar3} in \S \ref{bsar-plot} of the supplementary material, based on which it is clear that the performance of global-local shrinkage procedure is comparable with that of bsar. It is important to point out that bsar is also a shrinkage based method that allows for exact zeros in the coefficients in a transformed Gaussian process prior through a spike and slab specification.

\begin{figure}[htbp!]
	\centering
	\includegraphics[width=\textwidth]{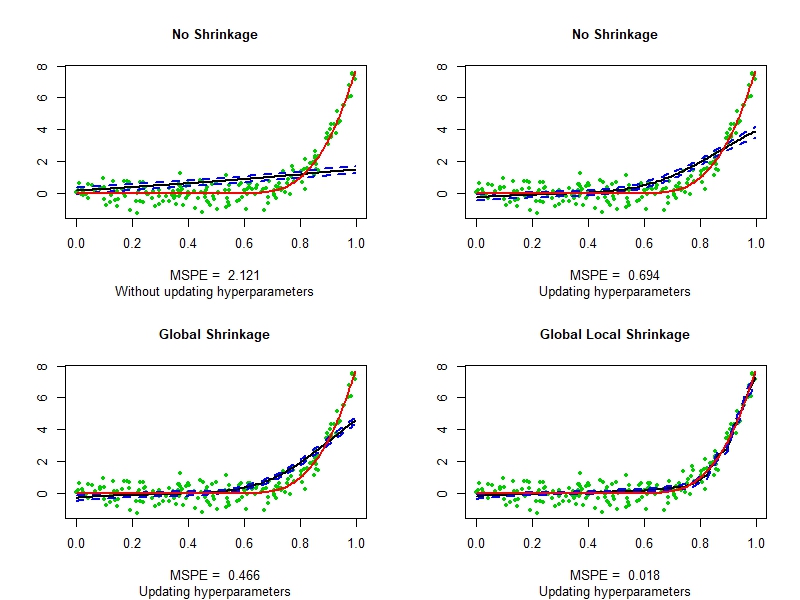}
	\caption{{\em Out-of-sample prediction accuracy for $f_1$ using the four variants. Red solid curve corresponds to the true function, black solid curve is the mean prediction, the region within two dotted blue curves represent 95\% pointwise prediction Interval and the green dots are $200$ test data points. MSPE values corresponding to each of the method are also shown in the plots.}}
	\label{shrink-f1}
\end{figure}

\begin{figure}[htbp!]
	\centering
	\includegraphics[width=\textwidth]{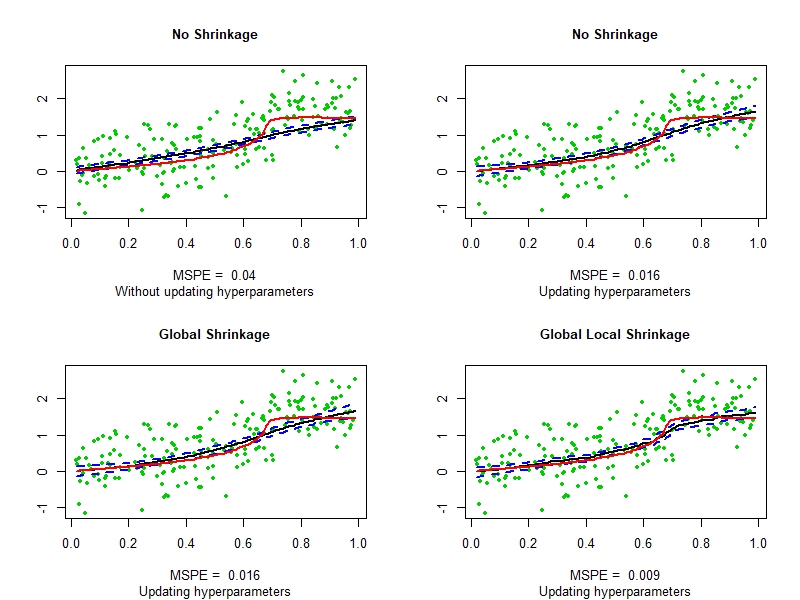}
	\caption{{\em Same as Figure \ref{shrink-f1}, now for the function $f_2$.}}
	\label{shrink-f2}
\end{figure}



\section{Discussion}\label{sec:disc}
A seemingly natural way to define a prior distribution on a constrained parameter space is to consider the restriction of a standard unrestricted prior to the constrained space. The conjugacy properties of the unrestricted prior typically carry over to the restricted case, facilitating computation. Moreover, reference priors on constrained parameters are typically the unconstrained reference prior multiplied by the indicator of the constrained parameter space \citep{sun1998reference}. Despite these various attractive properties, the findings of this article pose a caveat towards routine truncation of priors in moderate to high-dimensional parameter spaces, which might lead to biased inference. This issue gets increasingly severe with increasing dimension due to the concentration of measure phenomenon \citep{talagrand1995concentration,boucheron2013concentration}, 
which forces the prior to increasingly concentrate away from statistically relevant portions of the parameter space. A somewhat related issue with certain high-dimensional shrinkage priors has been noted in \cite{bhattacharya2016suboptimality}. 
Overall, our results suggest a careful study of the geometry of truncated priors as a useful practice. Understanding the cause of the biased behavior also suggests natural shrinkage procedures that can guard against such unintended consequences. We note that post-processing approaches based on projection \citep {lin2014bayesian} and constraint relaxation \citep{duan2020bayesian} do not suffer from this unintended bias. The same is also true for the recently proposed monotone \textsc{bart} (Bayesian Additive Regression Trees) \citep{chipman2010bart} method. 

It would be interesting to explore the presence of similar issues arising from truncations beyond the constrained regression setting. Possible examples include correlation matrix estimation and simultaneous quantile regression. Priors on correlation matrices are often prescribed in terms of constrained priors on covariance matrices, and truncated normal priors are used to maintain ordering between quantile functions corresponding to different quantiles, and this might leave the door open for unintended bias to creep in. 



\appendix
\section*{Appendix}
\section{Basis representation of \cite{maatouk2017gaussian}}\label{sec:app_posterior}
As our example which motivates the main results of this paper, we consider the more recent basis sequence of \cite{maatouk2017gaussian}. Let $u_j = j/(N-1), j = 0, 1, \ldots, N-1$ be equally spaced points on $[0, 1]$, with spacing $\delta_N = 1/(N-1)$. Let, 
\begin{align*}
h_j(x) = h\bigg(\frac{x-u_j}{\delta_N}\bigg), \quad \psi_j (x) = \int_0^x h_j(t)\,dt, \quad \phi_j (x) = \int_0^x \int_0^t h_j(u)\,dudt,
\end{align*} 
for $ j = 0, 1, \dots, N-1$, where $h(x) = (1-|x|)\, \mathbbm{1}_{[-1, 1]}(x)$ is the ``hat function" on $[-1, 1]$. For any continuous function $f: [0, 1] \to \mb R$, the function $\widetilde{f}(\cdot)
= \sum_{j=0}^{N-1} f(u_j) \, h_j(\cdot)$ approximates $f$ by linearly interpolating between the function values at the knots $\{u_j\}$, with the quality of the approximation improving with increasing $N$. With no additional smoothness assumption, this suggests a model for $f$ as $f(\cdot) = \sum_{j=0}^{N-1} \theta_{j+1} h_j(\cdot)$.

The basis $\{\psi_j\}$ and $\{\phi_j\}$ take advantage of higher-order smoothness. If $f$ is once or twice continuously differentiable respectively, then by the fundamental theorem of calculus, 
\begin{align*}
f(x) - f(0) = \int_0^x f{'}(t) dt, \quad f(x) - f(0) -  x f'(0) = \int_0^x \int_0^t f''(s) \, ds dt. 
\end{align*}
Expanding $f'$ and $f''$ in the interpolation basis as in the previous paragraph respectively imply the models
\begin{align}\label{eq:mc}
\underbrace{ f(x) = \theta_0 + \sum_{j=0}^{N-1} \theta_{j+1} \psi_j(x)}_{M}, \quad \underbrace{f(x) = \theta_0 + \theta^\ast x + \sum_{j=0}^{N-1} \theta_{j+1} \phi_j(x) }_{C}. 
\end{align}
Under the above, the coefficients have a natural interpretation as evaluations of the function or its derivatives at the grid points. For example, under (M), $f'(u_j) = \theta_{j+1}$ for $j = 0, 1, \ldots, N-1$, while under (C), $f''(u_j) = \theta_{j+1}$ for $j = 0, 1, \ldots, N-1$. 

\cite{maatouk2017gaussian} showed that under the representation (M) in \eqref{eq:mc}, $f$ is monotone non-decreasing {\em if and only if} $\theta_i \ge 0$ for all $i = 1, \ldots, N$. Similarly, under (C), $f$ is convex non-decreasing {\em if and only if} $\theta_i \ge 0$ for all $i = 1, \ldots, N$. The ability to {\em equivalently} express various constraints in terms of linear restrictions on the vector $\theta = (\theta_1, \ldots, \theta_N)^\T$ is an attractive feature of this basis not necessarily shared by other basis.

In either case, the parameter space $\m C$ for $\theta$ is the non-negative orthant $[0,\infty)^N$. 
If $f$ were unrestricted, a \textsc{gp} prior on $f$ would induce a dependent Gaussian prior on $\theta$. The approach of \cite{maatouk2017gaussian} is to restrict this dependent prior subject to the linear restrictions, resulting in a truncated normal prior. 

\section*{Supplementary Material}

In this supplementary document, we first collect all remaining technical proofs in the first two sections. \S \ref{sec:add_post} provides additional details on prior illustration, posterior computation, and posterior performance. In \S \ref{sec:mean} we formulate the concentration property of the posterior center $\mu_N$. Several auxiliary results used in the proofs are listed in \S \ref{sec:appd}. 

\section{Proofs of auxiliary results in the proof of Theorem \ref{thm:main}}\label{sec:proofs}
In this section, we provide proofs of Lemma \ref{slepian_gl_main} and Lemma \ref{upper_lower_small} that were used to prove Theorem \ref{thm:main} in the main manuscript. For any $N$-dimensional vector $a = [a_1, \ldots, a_d]^\T$ we denote its sub-vector $a_{[i_1\,:\,i_2]} = [a_{i_1}, \ldots, a_{i_2}]^\T$ for any $1\le i_1<i_2 \le d$. For two vectors $a$ and $b$ of the same length, let $a \ge b$ ($a \le b$) denote the event $a_i \ge b_i$ ($a_i \le b_i$) for all $i$. For two random variables $X$ and $Y$, We write $X \overset{\mathrm{d}}{=} Y$ if $X$ and $Y$ are identical in distribution.

%

\subsection{Proof of Lemma 1}\label{sec:slepian_gl}
For random vectors $X\sim \m N({\bf 0}, \Sigma_X)$ and $Y\sim \m N({\bf 0}, \Sigma_Y)$, to show $\bbP(\ell_1\le X_1 \le u_1,\,  X_2 \ge u_2,\ldots, X_d \ge u_d) \le \bbP (\ell_1\le Y_1 \le u_1,\,  Y_2\ge u_2, \ldots, Y_d \ge u_d)$, it suffices to show
\begin{equation}\label{eq:inq}
\begin{aligned}
\bbP(Y_1\ge u_1, Y_2\ge u_2,\dots, Y_d\ge u_d) - \bbP(X_1\ge u_1, X_2\ge u_2,\dots, X_d\ge u_d)  \\
\le  \bbP(Y_1\ge \ell_1, Y_2\ge u_2,\dots, Y_d\ge u_d)-\bbP(X_1\ge \ell_1, X_2\ge u_2,\dots, X_d\ge u_d). 
\end{aligned}
\end{equation}
We define $d$-dimensional indicator functions $G(x) = \mathbbm{1}_{[u_1,\infty)}(x_1) \prod_{j=2}^d \mathbbm{1}_{(u_j,\infty)}(x_j)$ and $F(x) = \mathbbm{1}_{[\ell_1,\infty)}(x_1) \prod_{j=2}^d \mathbbm{1}_{(u_j,\infty)}(x_j)$, then it is equivalent to show
\begin{align}\label{eq:inq_med}
\bbE \{G(Y)\} - \bbE \{G(X)\} \le \bbE \{F(Y)\} - \bbE \{F(X)\}. 
\end{align}
We now construct non-decreasing approximating functions of $G, F$ with continuous second order derivatives respectively. Let $\nu \in C^2({\mb R})$  be a non-decreasing twice differentiable function with $\nu(t)=0$ for $t\le 0$, $\nu (t) \in [0,1]$ for $t\in [0,1]$, and $\nu(t)=1$ for $t\ge 1$. Also, choose $\nu$ so that $\|\nu'\|_{\infty} < C$ for some universal constant $C>0$. For $\eta >0$, we define $m_{\eta}(x) =  \nu (\eta x)$. It is clear that $m_{\eta}(x)$ approximates $\mathbbm{1}_{[0,\infty)}(x)$ for large $\eta$. In fact, for any $x\neq 0$, $\lim_{\eta \to \infty}m_{\eta}(x)=\mathbbm{1}_{[0,\infty)}(x)$. 

Given the above, let $g^{\eta}_j (x_j)= \nu\{\eta (x_j-u_j)\}$ for $j=1,\dots, d$, and $f^{\eta}_1 = \nu\{\eta(x-\ell_1)\}$, $f^{\eta}_j = \nu\{\eta (x_j-u_j)\}$ for $j =2,\dots, d$.  Define
\begin{align*}
g^{\eta}(x) = \Pi_{j=1}^d g^{\eta}_j(x_j) \quad \mbox{and} \quad f^{\eta}(x) = \Pi_{j=1}^d f^{\eta}_j(x_j).
\end{align*}
It then follows that $g^{\eta}$ and $f^{\eta}$ provide increasingly better approximations of $G$ and $F$ as $\eta \to \infty$. It thus suffices to show 
\begin{align}\label{eq:bnd_app}
\bbE \{g^{\eta}(Y)\} - \bbE \{g^{\eta}(X)\} \le \bbE \{f^{\eta}(Y)\} - \bbE \{f^{\eta}(X)\}, 
\end{align}
for sufficiently large $\eta>0$ to be chosen later. We henceforth drop the superscript $\eta$ from $g$ and $f$ for notation brevity. 

We proceed to utilize an interpolation technique commonly used to prove comparison inequalities (see Chapter 7 of \cite{vershynin2018high}). We construct a sequence of interpolating random variables based on the independent random variables $X, Y$:
\begin{align*}
S_t = (1-t^2)^{1/2}X + tY, \quad t\in [0,1].
\end{align*}
Specifically, we have $S_0 = X$, $S_1 = Y$, and for any $t \in [0,1]$, $S_t \sim \m N({\bf 0}, \widetilde{\Sigma}_t)$ where $ \widetilde{\Sigma}_t = (1-t^2)\Sigma_X + t^2\Sigma_Y$. For any twice differentiable function $h$, we have the following identity 
\begin{align}\label{eq:int_parts}
\bbE \{h(Y)\} - \bbE \{h(X)\} = \int_0^1 \frac{d}{dt}\bbE \{h(S_t)\}\,dt. 
\end{align} 
Applying a multivariate version of Stein's lemma (Lemma 7.2.7 in \cite{vershynin2018high}) to the integrand in \eqref{eq:int_parts}, one obtains  
\begin{align}\label{eq:multi_stein}
\frac{d}{dt} \bbE \{h(S_t)\} &= t \sum_{i,j=1}^d \bbE\bigg[\{\bbE(Y_iY_j) - \bbE(X_iX_j)\}\frac{\partial^2 h}{\partial x_i \partial x_j} (S_t) \bigg]. 
\end{align}
To show \eqref{eq:bnd_app}, we define the difference $\Delta = [\bbE \{f(Y)\} - \bbE \{f(X)\}] - [\bbE \{g(Y)\} - \bbE \{g(X)\}]$.  We further decompose $\Delta$ as
\begin{align*}
\Delta &= [\bbE \{f(Y)\} - \bbE \{f(X)\}] - [\bbE \{g(Y)\} - \bbE \{g(X)\}]\\
&= \int_0^1 dt\, \bigg\{ \frac{d}{dt} \bbE \{f(S_t)\} - \frac{d}{dt} \bbE \{g(S_t)\} \bigg\} \\
&= \int_0^1 dt\, \bigg\{ t\, \sum_{i,j=1}^d \bbE \bigg[\, \{\bbE(Y_iY_j) - \bbE(X_iX_j)\} \bigg( \frac{\partial^2 f} {\partial x_i\partial x_j}(S_t) - \frac{\partial^2 g} {\partial x_i\partial x_j}(S_t) \bigg)\bigg]\bigg\}\\
&=2 \int_0^1 dt\, \bigg\{ t\, \sum_{j=2}^d \bbE \bigg[\, \{\bbE(Y_1Y_j) - \bbE(X_1X_j)\} \bigg( \frac{\partial^2 f} {\partial x_1\partial x_j}(S_t) - \frac{\partial^2 g} {\partial x_1\partial x_j}(S_t) \bigg)\bigg]\bigg\}\\
&~~~+ \int_0^1 dt\, \bigg\{ t\, \sum_{i,j=2}^d \bbE \bigg[\, \{\bbE(Y_iY_j) - \bbE(X_iX_j)\} \bigg( \frac{\partial^2 f} {\partial x_i\partial x_j}(S_t) - \frac{\partial^2 g} {\partial x_i\partial x_j}(S_t) \bigg)\bigg]\bigg\}\\
&= \Delta_1 + \Delta_2.
\end{align*}   
The second equation follows from \eqref{eq:int_parts} and the third equation follows from $\eqref{eq:multi_stein}$.  First we show $\Delta_1 \ge 0$. Since $\bbE(Y_1Y_j)\ge \bbE(X_1X_j)$ for all $j>1$, it suffices to show that for any fixed $t \in [0,1]$ and for any $j=2,\dots, d$, 
\begin{align*}
D_1 = \bbE \bigg( \frac{\partial^2 f} {\partial x_1\partial x_j}(S_t) - \frac{\partial^2 g} {\partial x_1\partial x_j}(S_t) \bigg) \ge 0.
\end{align*}
We consider a generic interpolating random variable $S \sim \m N\big(\mathbf{0}, \wt{\Sigma}\big)$ by dropping the $t$-subscript; let $\phi(s_1,\dots, s_d)$ denote its probability density function. Then we have 
\begin{align*}
D_1 &= \int_{-\infty}^{\infty} \dots \int_{-\infty}^{\infty} \{ f'_1(s_1)f'_j(s_j) - g'_1(s_1)g'_j(s_j) \} \,\Pi_{l\neq 1,j} f_l(s_l)\, \phi(s_1,\dots, s_d) \,ds_1\dots ds_d\\
&= \int_{-\infty}^{\infty} \dots  \int_{-\infty}^{\infty} \bigg[ \int_{-\infty}^{\infty} \{f'_1(s_1) - g'_1(s_1)\}\,\phi(s_1,\dots, s_N) \,ds_1 \bigg] \,f'_j(s_j)\,  \Pi_{l\neq 1,j} f_l (s_l)\,  ds_2\dots ds_d.
\end{align*}
To guarantee $D_1$ is non-negative we need the integral over $s_1$ to be non-negative. Based on the definition of $f_1$ and $g_1$, the integral over $s_1$ can be simplified to
\begin{align} \label{eq:int_s1}
&\int_{-\infty}^{\infty} \{f'_1(s_1) - g'_1(s_1)\} \, \phi(s_1,\ldots, s_N) \,ds_1 \nonumber \\
&= \int_{\ell_1}^{\ell_1+1/\eta} \big\{\eta\, \nu'\big(\eta (s_1-\ell_1)\big)\big\} \, \phi(s_1,\ldots, s_N) \,ds_1  - \int_{u_1}^{u_1+1/\eta} \big\{\eta\, \nu'\big(\eta (s_1-u_1)\big)\big\} \, \phi(s_1,\ldots, s_N) \,ds_1 \nonumber \\
&= \int_0^{1/\eta} \, \eta\, \nu'(\eta s_1) \{\phi(s_1+\ell_1, s_2, \ldots, s_N) - \phi(s_1+u_1, s_2,\ldots, s_N) \} ds_1.
\end{align}

Let us denote the inverse of the covariance matrix $\widetilde{\Sigma}$ as 
$$ \widetilde{\Sigma}^{-1} = \begin{bmatrix}
\widetilde{\Sigma}^{-1}_{11} &  \widetilde{\Sigma}^{-1}_{12}\\
 \widetilde{\Sigma}^{-1}_{21} & \widetilde{\Sigma}^{-1}_{22}\\
\end{bmatrix},$$ 
where $\widetilde{\Sigma}^{-1}_{11}$ is a scalar. To check the non-negativity of the last line in \eqref{eq:int_s1}, we now estimate the term 
\begin{align*}
\frac{\phi(s_1+\ell_1,s_2,\ldots, s_d)}{\phi(s_1+u_1,s_2,\ldots,s_d)} = e^{\{(u_1^2-\ell_1^2)+2s_1\,(u_1-\ell_1)\}\, \widetilde{\Sigma}^{-1}_{11}/2+\,(u_1-\ell_1)\,  \widetilde{\Sigma}^{-1}_{12}\,\tilde{s}_2},
\end{align*}
where $\tilde{s}_2 = (s_2,\dots, s_d)^{\T}$. Since $s_j \in [0,1/\eta]$, we have $s_1\,(u_1-\ell_1)\}\, \widetilde{\Sigma}^{-1}_{11}>0$. We denote $\tilde{\rho} = \max \{\widetilde{\Sigma}^{-1}_{12}\}$ as the largest element of $\widetilde{\Sigma}^{-1}_{12}$. Then, one can choose $\eta$ large enough such that 
\begin{align*}
(u_1+\ell_1)\,\widetilde{\Sigma}^{-1}_{11} - 2 (d-1)\tilde{\rho}/\eta \ge 0, 
\end{align*}
to guarantee $D_1 \ge 0$. For example  $\eta = 4 (d-1)\,\tilde{\rho}\, \widetilde{\Sigma}_{11} /(u_1+\ell_1)$ satisfies the above inequality. 

Now we show $\Delta_2 \ge 0$. We have $\bbE(Y_iY_j)\ge \bbE(X_iX_j)$ for all $i,j = 2,\dots, d$. For any $i,j \ge 2$, for any fixed $t \in [0,1]$, we define 
\be 
D_2 = \bbE \bigg( \frac{\partial^2 f} {\partial x_i\partial x_j}(S_t) - \frac{\partial^2 g} {\partial x_i\partial x_j}(S_t) \bigg) = \bbE \{ (f_1 - g_1)f'_i\,g'_j\, \Pi_{k\neq 1,i,j}f_k \}.
\ee
Since $ f_1- g_1\ge 0$, and $f'_j \ge 0$ for all $j>1$, it follows that $D_2 \ge 0$ and thus $\Delta_2 \ge 0$.  Combining with the non-negativity of $\Delta_1$ completes the proof of Lemma \ref{sec:slepian_gl}. 
\subsection{Proof of Lemma \ref{upper_lower_small}} \label{upper_new} 
For $X\sim \m N (\mathbf{0}, \Sigma_d(\rho))$ with $\rho\in(0,1)$, we will repeatedly use its equivalent expression 
\begin{align}\label{eq:X_2}
X_i = \rho^{1/2}\, w + (1- \rho)^{1/2}\, W_i \quad (i=1,\ldots,N),
\end{align}
where $w, W_i$'s are independent standard normal variables. \\[1ex]

\noindent {\bf Proof of the upper bound.}  We recall $\bar{\rho} = (1-\rho)/\rho$.  
For any fixed $\delta>0$ and $\alpha \in (0,1)$, we have 
\begin{align}\label{eq:P}
&\bbP(0\le X_1<\delta,X_2\ge 0, \ldots, X_d \ge 0) \\ \nonumber
&=\bbP\Big(0\le {\rho}^{1/2}\,w + {(1-\rho)}^{1/2}\, W_1 \le \delta, w \ge  {\bar{\rho}}^{1/2}\max_{2\le i\le d}W_i\Big)\\\nonumber
&=\bbP\Big(\Big\{0 \le {\rho}^{1/2}\,w + {(1-\rho)}^{1/2}\, W_1 \le \delta, w \ge {\bar{\rho}}^{1/2}\max_{2\le i\le d}W_i\Big\}\\\nonumber
&~~~~\cup \Big[\max_{i\le d}W_i \ge \{2(1-\alpha)\log (d-1)\}^{1/2}\Big] \cup \Big[\max_{i\le d}W_i \le \{2(1-\alpha)\log (d-1)\}^{1/2}\Big ]\Big)\\\nonumber
&\le \bbP\Big[0\le  {\rho}^{1/2}\,w + {(1-\rho)}^{1/2}\, W_1 \le \delta, w \ge \{2\,\bar{\rho}\,(1-\alpha)\log (d-1)\}^{1/2}\Big] \\\nonumber
&~~~~+ \bbP\Big[\max_{i\le d}W_i \le \{2(1-\alpha)\log (d-1)\}^{1/2}\Big]\\\nonumber
&=P_1 +P_2.
\end{align}
First, we estimate $P_1$ in \eqref{eq:P}. By applying the equivalent expression of $X$ in \eqref{eq:X_2}, we have  
\begin{align*}
P_1 &= \bbP\bigg[ W_1 \in \bigg\{-\bigg(\frac{\rho}{1-\rho}\bigg)^{1/2} w,  \delta/(1-\rho)^{1/2}-\bigg(\frac{\rho}{1-\rho}\bigg)^{1/2} w\bigg\} \mid w \ge  \{2\,\bar{\rho}\,(1-\alpha)\log (d-1)\}^{1/2}\bigg]\\\nonumber
&~~~~~\bbP\big[ w \ge \{2\,\bar{\rho}\,(1-\alpha)\log (d-1)\}^{1/2}\,\big]\\\nonumber
&\le \bbP\big[ W_1 \in \big(  -\{2(1-\alpha)\log (d-1)\}^{1/2},  \delta(1-\rho)^{-1/2}-\{2(1-\alpha)\log (d-1)\}^{1/2}\,\big)\big]\\\nonumber
&~~~~~\bbP\big( w \ge \{2\,\bar{\rho}\,(1-\alpha)\log (d-1)\}^{1/2}\,\big)\\\nonumber
&\le \delta(2\pi)^{-1/2}\, \exp\big(-\big[ \delta(1-\rho)^{-1/2}-\{2(1-\alpha)\log (d-1)\}^{1/2}\big]^2/2\big)\\\nonumber
&~~~~~\{2\,\bar{\rho}\,(1-\alpha)\log (d-1)\}^{-1/2}\,\exp\big(- \{2\,\bar{\rho}\,(1-\alpha)\log (d-1)\}^2/2\big).
\end{align*}
The last inequality follows from Lemma \ref{lem:l4} in  Appendix \ref{sec:appd}. 

Now we move to estimate the term $P_2$ in \eqref{eq:P}. We have, 
\begin{align*}
 \bbP\big[\max_{i\le d} W_i \le \{2\,(1-\alpha)\log d\}^{1/2}\big] &= \big(1-\bbP\big[Z\ge \{2\,(1-\alpha)\,\log d\}^{1/2}\,\big]\big)^d\\
& \le \exp\big(-d\, \bbP\big[Z\ge \{2\,(1-\alpha)\,\log d\}^{1/2}\,\big]\big) 
 \le \exp(-d\,^{\alpha}),
\end{align*}
where $Z\sim \m N(0,1)$.

Combining bounds for $P_1$ and $P_2$, we obtain 
\begin{align*}
\bbP(0\le X_1<\delta,X_2\ge 0, \ldots, X_d \ge 0) \le \delta\,  \{2\bar{\rho}\,(1-\alpha)\log (d-1)\}^{-1/2}\, (d-1)^{-(1-\alpha)/\rho} + \exp(-d\,^\alpha),
\end{align*}
for any $\alpha \in (0,1)$. We then complete the proof of the upper bound.

\noindent {\bf Proof of the lower bound.} We provide a more general result for the lower bound. We show that for any scalar $a\ge0$, we have 
\begin{align}\label{eq:lower_general}
\bbP(X \ge a \mathbf{1}_d) \ge \frac{a {\rho}^{-1/2} +  (2\,\bar{\rho}\, \log N)^{1/2}}{\big\{a {\rho}^{-1/2} + (2\,\bar{\rho}\, \log d)^{1/2}\big\}^2+1} \exp  \bigg[- \frac{1}{2} \bigg\{a {\rho}^{-1/2} + (2\,\bar{\rho}\, \log d)^{1/2}\bigg\}^2\bigg],
\end{align} 
where recall that $\mathbf{1}_N$ denotes a $N$-dimensional vector of ones.  By taking $a =0$ leads to the desired lower bound in Lemma \ref{upper_lower_small}. 

Now we prove the lower bound in \eqref{eq:lower_general}. First,  
\begin{align}\label{res1}
\bbP( X \ge a\mathbf{1}_d) &= \bbP \big(\rho^{1/2}\, w + (1-\rho)^{1/2}\, W_i \geq a, \, \mbox{for}\, i =1,\dots,d\big)  \\\nonumber
&= \bbE \big(\,\bbP\big[w\ge {\rho}^{-1/2}\{a - (1-\rho)^{1/2} \,W_i \},\, i = 1,\dots,d \mid {W_1,\ldots, W_d}\big]\big)\\\nonumber
& \overset{\mathrm{(i)}}{=}  \bbE \big(\,\bbP \big[w\ge \rho^{-1/2}\big\{a +(1-\rho)^{1/2} \,\max_i W_i \big\} \mid W_1,\ldots, W_N\big]\big)\\\nonumber 
&=  \bbE \bigg\{1-\Phi\bigg(a\rho^{-1/2}+ {\bar{\rho}}^{1/2}\max_i W_i\bigg)\bigg\}, 
\end{align}
where $W = [W_1,\dots, W_d]^\T$. Here, $(\mbox{i})$ holds since $-W_i \overset{\mathrm{d}}{=} W_i$ for $i=1,\dots, d$ and $\max_{i\le d} (-W_i) \overset{\mathrm{d}}{=} \max_{i \le d}(W_i)$.\\  
\
\\
We now proceed to lower bound the right hand side of the last equation in \eqref{res1}. To that end, we define $g(a, b) = 1 - \Phi( a\rho^{-1/2} + {\bar{\rho}}^{1/2}\,b)$, where $g: \mathbb{R}_+\times \mathbb{R} \to [0,1]$. Importantly, $g$ is non-increasing function of $a,b$ for $a,b \in \mathbb{R}$, and $g$ is a convex function of $(a, b)$ for $a,b>0$.  For any fixed $a >0$, since $g(a, \max_i W_i)$ is non-increasing in $\max_i W_i$, we have $g(a, \max_i W_i) \ge g(a, \max_i |W_i|)$. We then apply Jensen's inequality, 
\begin{align*}
{ \bbE} \Big\{g\Big(a, \max_{1\le i\le d} |W_i|\Big)\Big\} \ge g\Big\{a, \bbE \Big(\max_{1\le i\le d} |W_i|\Big)\Big\} \ge g\big\{a, (2 \log d)^{1/2}\big\}.
\end{align*}
The last inequality holds by applying Lemma \ref{lem:l3} in Appendix \ref{sec:appd}. To lower bound $g\big\{a, (2 \log d)^{1/2}\big\}$ we apply Lemma \ref{lem:l4} in  Appendix \ref{sec:appd}. Eventually, we obtain
\begin{align*}
\bbE\, \Big\{g\Big(a, \max_{1\le i\le d} |W_i|\Big)\Big\} \ge \frac{a\rho^{-1/2} +  (2\bar{\rho}\log d)^{1/2}}{\big\{a\rho^{-1/2} + (2\bar{\rho}\log d)^{1/2}\big\}^2 +1}\, \exp\big[-\big\{a\rho^{-1/2} +  (2\bar{\rho}\log d)^{1/2}\big\}^2/2\big]. 
\end{align*}
This completes the proof for the lower bound.

\section{Remaining proofs from the main document}

\subsection{Proof of the Proposition \protect \ref{PROP:MARG_K}}\label{sec:proof_MARG}

Now we derive the $k$-dimensional marginal density function.  We denote $\theta^{(k)}= (\theta_1,\dots,\theta_k)^{\T}$ and $\theta^{(N-k)} = (\theta_{k+1},\dots, \theta_N)^{\T}$. We partition $\Sigma_N$ into appropriate blocks as
$$\Sigma_N= \begin{bmatrix}
\Sigma_{k,k} &  \Sigma_{k,N-k}\\
\Sigma_{N-k,k} & \Sigma_{N-k,N-k}\\
\end{bmatrix}.
$$
We also partition its inverse matrix $\widetilde{\Sigma}_N$, 
$$\widetilde{\Sigma}_N = \begin{bmatrix}
\widetilde{\Sigma}_{k,k} &  \widetilde{\Sigma}_{k,N-k}\\
 \widetilde{\Sigma}_{k,N-k} & \widetilde{\Sigma}_{N-k,N-k}\\
\end{bmatrix}.$$
Then the $k$-dimensional marginal $\widetilde{p}_{k,N}(\theta_1,\dots,\theta_k) = $
\begin{align*}
& \bigg(\frac{1}{2\pi}\bigg)^{N/2}\,\{\mbox{det}(\Sigma)\}^{-1/2}\,\int_0^\infty\dots \int_0^\infty \exp\big\{-\big({\theta^{(k)}}^{\T}\widetilde{\Sigma}_{k,k}\, \theta^{(k)}  \\
&~~~~-2{\theta^{(k)}}^{\T}\,\widetilde{\Sigma}_{k,N-k}\,\theta^{(N-k)}+{\theta^{(N-k)}}^\T \,\widetilde{\Sigma}_{N-k,N-k}\,\theta^{(N-k)}\big)/2\big \}\, d \theta^{(N-k)}\\
&=\bigg(\frac{1}{2\pi}\bigg)^{k/2} \exp \big\{-{\theta^{(k)}}^{\T}\widetilde{\Sigma}_{k,k}\theta^{(k)}/2\big\} \cdot \Pi_{i=1}^k\ind_{[0,\infty)}(\theta_i)\bigg(\frac{1}{2\pi}\bigg)^{(N-k)/2}  \{\mbox{det}(\widetilde{\Sigma}_{N-k,N-k})\}^{-1/2}  \\
&~~~~\cdot \int_0^\infty \dots \int_0^\infty  \exp\big\{-\| \widetilde{\Sigma}_{N-k,N-k}^{\frac{1}{2}}\,\big(\theta^{(N-k)} -\Sigma_{N-k,k}\,\Sigma^{-1}_{k,k}\,\theta^{(k)}\big)\|^2 /2 \big\} d\theta^{(N-k)}\\
&=\bigg(\frac{1}{2\pi}\bigg)^{k/2} \exp \{-{\theta^{(k)}}^{\T} \widetilde{\Sigma}_{k,k}\,\theta^{(k)}/2\} \,\bbP(\widetilde{X}_{N-k} \le \Sigma_{N-k,k} \,\Sigma^{-1}_{k,k}\, \theta^{(k)})\cdot \Pi_{i=1}^k \ind_{[0,\infty)}(\theta_i).
\end{align*}
where  
\begin{align*}
\widetilde{\Sigma}_{k,k} &= \Sigma^{-1}_{k,k}+\Sigma^{-1}_{k,k}\,\Sigma_{k,N-k}\,\widetilde{\Sigma}_{N-k,N-k}\,\Sigma_{N-k,k}\,\Sigma^{-1}_{k,k},\\
\widetilde{\Sigma}_{k,N-k} &=\Sigma^{-1}_{k,k}\,\Sigma_{k,N-k}\,\widetilde{\Sigma}_{N-k,N-k}, \\
\widetilde{\Sigma}_{N-k,N-k} &= (\Sigma_{N-k,N-k} - \Sigma_{N-k,k}\,\Sigma^{-1}_{k,k}\,\Sigma_{k,N-k})^{-1},
\end{align*}
and $\widetilde{X}_{N-k} \sim \m N_{N-k} (\mathbf{0}_{N-k}, \widetilde{\Sigma}^{-1}_{N-k,N-k})$. 

\subsection{Proof of Proposition \ref{prop:app}}\label{sec:band_proof}
 We first introduce some notations that are used in the proof. For a $N \times N$ matrix $A$, we denote $\lambda_j(A)$ as its $j$th eigenvalue, and denote $\lambda_{\min}(A)$ and $ \lambda_{\max}(A)$ as the minimum and maximum of eigenvalues, respectively. For a matrix $A$, we define its operator norm as $\|A\| = \{\lambda_{\max}(A^\T A)\}^{1/2}$. For two quantities $a,b$, we write $a \asymp b$ when  
 $a/b$ can be bounded from below and above by two finite constants. 
   
We repeatedly apply Newmann series and Lemma \ref{band_inv} in Appendix \ref{sec:appd} to construct the approximation matrix to the posterior scale matrix $\Sigma_N$.  Under Assumption \ref{ass:band}, we have the prior covariance matrix $\Omega_N \in \m M(\lambda_0, \alpha, k)$ for some universal constants $\lambda_0, \alpha, k>0$. Then for any $\epsilon\in (0,\lambda_0/2)$, by choosing $r \ge \log(C/\epsilon)/\alpha$, one can find a $r$-banded symmetric and positive definite matrix $\Omega_{N,r}$ such that 
\begin{align}\label{r_op}
\| \Omega_N - \Omega_{N,r} \| \le \epsilon.
\end{align}
Now we let $M = \lambda_{\max}(\Omega_{N,r})$ and $m = \lambda_{\min}(\Omega_{N,r})$. Given \eqref{r_op}, we have 
\begin{align}\label{band1_eigen}
\lambda_0-\epsilon \le m \le M \le 1/\lambda_0 +\epsilon.
\end{align} 
By choosing $\xi = 2/(M+m)$, simple calculation gives $\|I_N - \xi\,\Omega_{N,r}\| < 1$. We now apply Newmann series to construct a polynomial of $\Omega_{N,r}$ of degree $n_1$, defined as   
$\widetilde{\Omega}^{-1} = \xi \, \sum_{j=0}^{n_1}(I - \xi\,\Omega_{N,r})^j$, for some integer $n_1>0$ to be chosen later. Applying Lemma \ref{band_inv} in  Appendix \ref{sec:appd}, we have 
\begin{align}\label{band2_op}
\| \Omega_{N,r}^{-1} - \widetilde{\Omega}^{-1} \| \le \kappa_0^{n_1+1}/(\lambda_0-\epsilon), 
\end{align}
where $\kappa_0 = (M-m)/(M+m)$. Applying Lemma \ref{band_inv} we guarantee $\widetilde{\Omega}^{-1}$ is $(n_1\,r)$-banded and positive definite.  Combining results in \eqref{band1_eigen} and \eqref{band2_op}, we have
\begin{align}\label{eq:eigen_omega}
\lambda_0/(1+\lambda_0\epsilon) - \kappa_0^{n_1+1}/(\lambda_0-\epsilon) \le\lambda_{\min}(\widetilde{\Omega}^{-1})\le \lambda_{\max}(\widetilde{\Omega}^{-1}) \le 1/(\lambda_0-\epsilon) + \kappa_0^{n_1+1}/(\lambda_0-\epsilon).
\end{align}
Now we let $\widetilde{\Sigma}^{-1}=\widetilde{\Omega}^{-1} + \Phi^\T\Phi$. Under Assumption \ref{ass:basis} we have $\widetilde{\Sigma}^{-1}$ is $k$-banded with $k = \max\{n_1\,r, q\}.$ We then define $\tilde{\lambda}_1 = \lambda_{\max}(\widetilde{\Sigma}^{-1})$ and $\tilde{\lambda}_N = \lambda_{\min}(\widetilde{\Sigma}^{-1})$. Thus, given \eqref{eq:eigen_omega}, we have 
 \begin{align*}
C_1\,(n/N) + \lambda_0/(1+\lambda_0\epsilon) - \kappa_0^{n_1+1}/(\lambda_0-\epsilon) \le \tilde{\lambda}_N \le\tilde{\lambda}_1 \le C_2\,(n/N) + 1/(\lambda_0-\epsilon) + \kappa_0^{n_1+1}/(\lambda_0-\epsilon),
\end{align*}
for constants $0<C_1 <C_2 <\infty$ in Assumption \ref{ass:band}. 

We first consider the case where $N/n \to a$ for some constant $a\in (0,1)$, as $n,N \to \infty$. For sufficiently large $n,N$, we obtain 
\begin{align}\label{eigen_lim_I}
C'_1 a + \lambda_0/(1+\lambda_0\epsilon) \le \tilde{\lambda}_N \le\tilde{\lambda}_1 \le C'_2\,a + 1/(\lambda_0-\epsilon),
\end{align}
for constants $C'_1,C'_2$ satisfying $C'_1<C_1$ and $C_2<C'_2$.

Secondly, we consider the case where $N/n\to 0$ as $n,N\to \infty$. In this case,  $n/N$ dominates in the eigenvalues of $\widetilde{\Sigma}^{-1}$. Thus, for sufficiently large $n,N$, we have
\begin{align}\label{eigen_lim_II}
C_1 \,(n/N) \le \tilde{\lambda}_N \le\tilde{\lambda}_1 \le C_2\,(n/N).
\end{align}
Now we apply Lemma \ref{band_inv} one more time to construct the approximation matrix to the inverse of $\widetilde{\Sigma}^{-1}$. Again, by taking $\gamma = 2/(\tilde{\lambda}_1 + \tilde{\lambda}_N)$, we have $\| I_N -\gamma\, \widetilde{\Sigma}^{-1}\| <1$. Now we define $\Sigma'  =\gamma\,\sum_{j=0}^{m_1} (I_N -\gamma\, \widetilde{\Sigma}^{-1})^j$ for some positive integer $m_1$. Also, it follows 
\begin{align}\label{app_sigma}
\|\widetilde{\Sigma} - \Sigma' \| \le \tilde{\kappa}^{m_1+1}/{\tilde{\lambda}_N}, 
\end{align}
where $\tilde{\kappa} = (\tilde{\lambda}_1-\tilde{\lambda}_N)/(\tilde{\lambda}_1+\tilde{\lambda}_N)$. By construction $\Sigma'$ is $(m_1\,k)$-banded. 

Now we estimate $\tilde{\kappa}$. For large enough $N,n$ in the first case, we can upper bound 
$$\tilde{\kappa} \le \kappa_1 = \frac{(C'_2-C'_1)\,a + 1/(\lambda_0-\epsilon) - \lambda_0/(1+\lambda_0\epsilon)} {(C'_2+C'_1)\,a + 1/(\lambda_0-\epsilon) + \lambda_0/(1+\lambda_0\epsilon)}.$$
The inequality holds since the map $x \mapsto (1-x)/(1+x)$ is non-increasing in $x \in (0,1)$. Combing this with the result in \eqref{eigen_lim_I} and taking $x = \tilde{\lambda}_N/\tilde{\lambda}_1$ leads to the expression of $\kappa_1$. Based on \eqref{app_sigma}, we have $\|\widetilde{\Sigma} - \Sigma' \| \le \kappa_1^{m_1+1}/\{C'_1 a + \lambda_0/(1+\lambda_0\epsilon)\}$.  For $N,n$ in the second case, following a similar line of argument, we have $\|\widetilde{\Sigma} - \Sigma'\| \le \tilde{\kappa}^{m_1+1}\, N/(C_1n)$ with $\tilde{\kappa} = (C_2-C_1)/(C_2+C_1)$.

We recall the posterior scale matrix $\Sigma_N = (\Omega_N^{-1}+\Phi^\T\Phi)^{-1}$. Then we have 
\begin{align*}
\| \Sigma_N - \Sigma'\| &\le \| \Sigma_N - \widetilde{\Sigma} \| + \|\widetilde{\Sigma} -  \Sigma' \| \\ \nonumber
&\le  \| \Sigma_N\| ( \| \Omega_N^{-1} - \Omega_{N,r}^{-1} \|  + \|\Omega_{N,r}^{-1} - \widetilde{\Omega}^{-1}\| )\|\widetilde{\Sigma}\| + \|\widetilde{\Sigma} -  \Sigma'\| \\ \nonumber
& \le \| \Sigma_N\| \|\widetilde{\Sigma}\| (c_1\,\epsilon + c_2\, \kappa_0^{n_1+1}) + \|\widetilde{\Sigma} - \Sigma' \| 
\end{align*}
where $c_1 = \|\Omega^{-1}\|\|\Omega_{N,r}^{-1}\|$ and $c_2 = 1/(\lambda_0-\epsilon)$. The first inequality follows from the triangular inequality and the second inequality follows from the identity $\|A^{-1}-B^{-1}\|= \|A^{-1}\| \|A-B\|\|B^{-1}\|$ for invertible matrices $A,B$. The last inequality follows from results in \eqref{r_op} and \eqref{band2_op}. 

For $N,n$ in the first case, $\|\Sigma_N\|$ and $ \|\widetilde{\Sigma}\| $ are upper bounded by some constants that are free of $n,N$ given \eqref{eigen_lim_I}.  Then we obtain 
\begin{align*}
\| \Sigma_N - \Sigma'\| \le  C'(\epsilon+\kappa_0^{n_1+1} + \, \kappa_1^{m_1+1}),
\end{align*} 
where $C' = \max \{c_1,c_2, C'_1\,a +\lambda_0/(1+\lambda_0\epsilon)\}/\{C'_1\,a +\lambda_0/(1+\lambda_0\epsilon)\}^2$. 

For $N,n$ in the second case,  for sufficiently large $N,n$ we have $\|\Sigma_N\| \asymp (N/n)$ given \eqref{eigen_lim_II}. Then we have
\begin{align*}
 \| \Sigma_N - \Sigma'\| \le C''\, \{(N/n)^{2}(\epsilon +\kappa_0^{n_1+1}) + (N/n) \tilde{\kappa}^{m_1+1}\},
 \end{align*}
where $C'' =C^{-2}_1\,\max\{c_1,c_2, C_1\}$. Letting $\kappa= \max\{\kappa_0, \kappa_1,\tilde{\kappa}\}$, $n_0 =\min\{n_1,m_1\}$, and $\delta_{\epsilon,\kappa} = (\epsilon +\kappa^{n_0+1})\max\{(N/n),(N/n)^2\}$ yields the result in Proposition \ref{prop:app}.

\section{Additional details on the numerical studies}\label{sec:add_post}

\subsection{Prior draws}\label{app:prior-draw}

We consider equation \eqref{shrink_eq} and the prior specified in section \ref{sec:meth}. Prior samples on both $\theta$ and $\xi$ of dimension $N = 100$ were drawn. Figure \ref{prior-draw} shows prior draws for the first and third components of both $\theta$ and $\xi$. 

\begin{figure}
	\centering
	\includegraphics[width=0.8\textwidth]{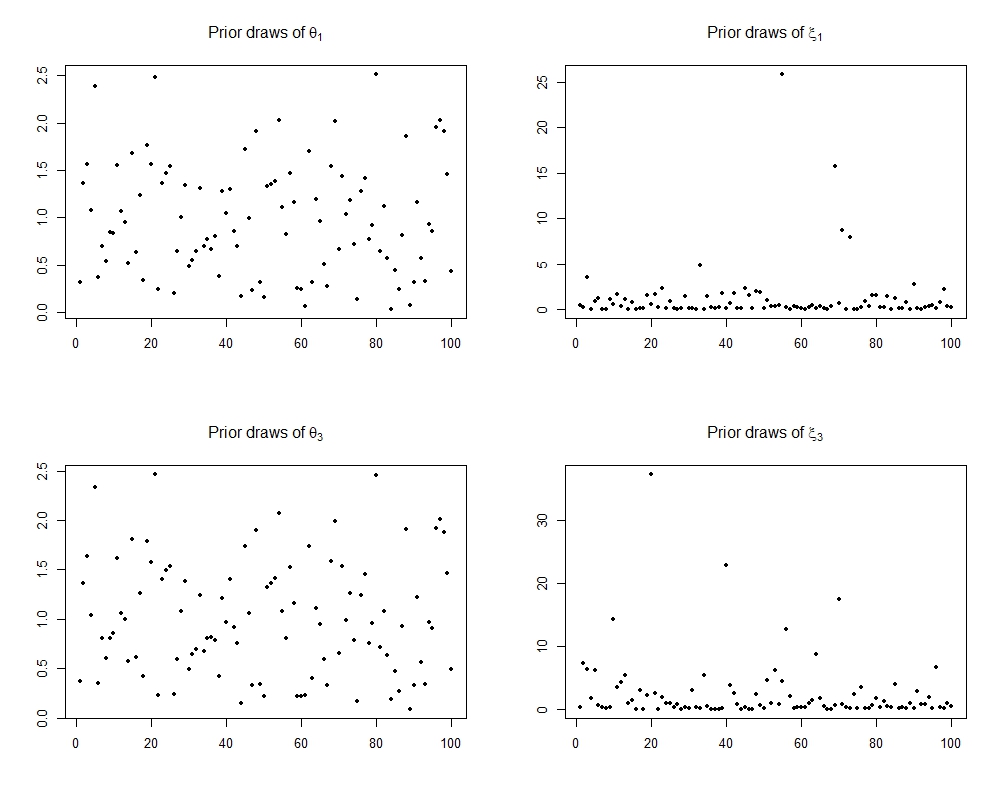}
	\caption{{\em Showing prior draws from distribution of $\theta$ (left panel) and $\xi$ (right panel). Top and bottom panels correspond to first and third components respectively, for both $\theta$ and $\xi$.}}
	\label{prior-draw}
\end{figure}

\subsection{Posterior Computations} \label{post-sam}

We now consider model \eqref{model} and the prior specified in section \ref{sec:meth}. Then the full conditional distribution of $\theta$
\begin{align*}
\pi (\theta \mid Y, \xi_0, \lambda, \tau, \sigma) \propto \exp \bigg\{ - \frac{1}{2 \sigma^2} \|\widetilde{Y} - \Psi \Lambda \theta \|^2 \bigg \} \, \exp \bigg\{ - \frac{1}{2 \tau^2} \theta^\T K^{-1} \theta \bigg \} \, \ind_{\mc C_{\theta}}(\theta)
\end{align*}
can be approximated by 
\small \begin{eqnarray}
\pi (\theta \mid Y, \xi_0, \lambda, \tau, \sigma) \propto \exp \bigg\{ - \frac{1}{2 \sigma^2} \|\widetilde{Y} - \Psi_{\lambda} \, \theta \|^2 \bigg \} \, \exp \bigg\{ - \frac{1}{2 \tau^2} \theta^\T K^{-1} \theta \bigg \} \, \bigg\{ \prod_{j=1}^{N+1} \frac{e^{\eta \, \theta_j}}{1 + e^{\eta \, \theta_j}} \bigg\} \nonumber \\
= \bigg[ \exp \bigg\{ - \frac{1}{2 \sigma^2} \|\widetilde{Y} - \Psi_{\lambda} \, \theta \|^2 \bigg \} \, \bigg\{ \prod_{j=1}^{N+1} \frac{e^{\eta \, \theta_j}}{1 + e^{\eta \, \theta_j}} \bigg\} \bigg] \, \exp \bigg\{ - \frac{1}{2 \tau^2} \theta^\T K^{-1} \theta \bigg \} \nonumber
\end{eqnarray}
\normalsize
where $\eta$ is a large valued constant, $\widetilde{Y} = Y - \xi_0 \mr 1_n$ and $\Psi_{\lambda} = \Psi  \Lambda$. The above is same as equation (5) of \cite{ray2019efficient} and thus falls under the framework of their sampling scheme. For more details on the sampling scheme and the approximation, one can refer to \cite{ray2019efficient}. 

Note that $\lambda_j \sim \mc C_{+}(0,1), \,\, j=1, \ldots, N$, can be equivalently given by $\lambda_j \mid w_j \sim \mc N(0, w_j^{-1}) \ind(\lambda_j > 0) \, , \,\, w_j \sim \mc G(0.5,0.5) \, , \,\, j = 1, \ldots, N$. Thus the full conditional distribution of $\lambda$ can be approximated by:
\small
\begin{align*}
\pi (\lambda \mid Y, \xi_0, \theta, w, \tau, \sigma) \propto  \bigg[ \exp \bigg\{ - \frac{1}{2 \sigma^2} \|\widetilde{Y} - \Psi_{\theta} \, \lambda \|^2 \bigg \} \, \bigg\{ \prod_{j=1}^{N+1} \frac{e^{\zeta \, \lambda_j}}{1 + e^{\zeta \, \lambda_j}} \bigg\} \bigg] \, \exp \bigg\{ - \frac{1}{2} \lambda^\T W \lambda \bigg \}
\end{align*}
\normalsize
where $\zeta$ plays the same role as $\eta$, $w = (w_1, \ldots, w_{N})^\T$, $W=\mbox{diag}(w_1, \ldots, w_{N})$, $\Psi_{\theta} = \Psi  \Theta$ and $\Theta = \mbox{diag}(\theta_1, \ldots, \theta_{N})$. Thus, $\lambda$ can be sampled efficiently using algorithm proposed in \cite{ray2019efficient}.

\subsection{Performance of bsar} \label{bsar-plot}  

Consider the simulation set-up specified in section \ref{sec:meth}.  Figure \ref{bsar3} shows the out-of-sample prediction performance of bsar, developed by \cite{lenk2017bayesian}, and implemented by the \textbf{{\fontfamily{qcr} \selectfont R}} package \textbf{{\fontfamily{qcr} \selectfont bsamGP}}.

\begin{figure}[htbp!]
	\centering
	\includegraphics[width=\textwidth]{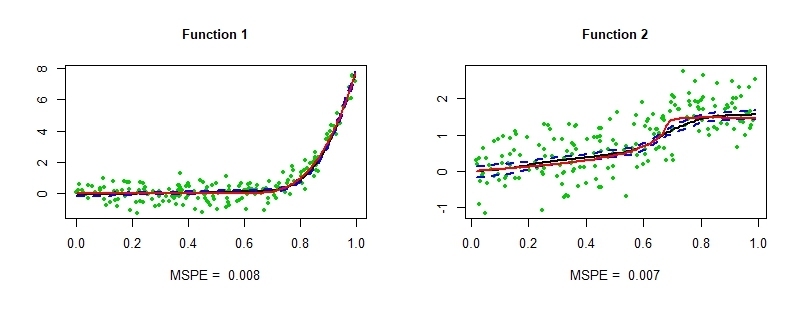}
	\caption{{\em Figure portraying out-of-sample prediction accuracy using bsar for $f_1$ and $f_2$. Red solid curve corresponds to the true function, black solid curve is the mean prediction, the region within two dotted blue curves represent 95\% pointwise prediction Interval and the green dots are $200$ test data points. \textsc{mspe} values corresponding to each of the method are also shown in the plots.}}
	\label{bsar3}
\end{figure}

\section{Concentration result of the posterior mean $\mu$}\label{sec:mean}
We start by introducing some new notations and assumptions. For two variables $X, Y$, we denote the conditional probability measure, conditional expectation and conditional variance of $Y$ given $X$ as $\bbP_{Y\mid X}$ $\bbE_{Y\mid X}$, and $\Var_{Y\mid X}$, respectively. For two quantities $a,b$, we write $a \succsim b$ when $a$ is bounded below by a multiple of $b$. For matrices $A, B$ of the same size, we say $A\le B$ if $B-A$ is positive semi-definite. In the following, we state the assumptions on the basis choice and prior preferences. 

\begin{assumption}\label{ass_basis}
We assume the number of basis $N$ and sample size $n$ satisfy $N/n \to 0$ as $N,n \to \infty$. Also, we assume the basis matrix $\Phi_{n\times N}$ satisfies 
$$ c_1 (n/N)\, I_N\le \Phi^\T\Phi \le c_2 (n/N)\, I_N, $$
 for some constants $0<c_1< c_2<\infty$.
\end{assumption}
For an example of basis that satisfies Assumption \ref{ass_basis}, we take a $q$th ($q\ge 2$) order B-Spline basis function associated with $N-q$ knots. Moreover, under mild conditions, it can be shown that the optimal order of the number of basis $N \asymp n^{c}$ for some $c\in (0,1)$ in the regression setting \citep{yoo2016supremum}.
\begin{assumption}\label{prior_cov}
For the prior distribution $\theta \sim \m N(\mathbf{0}, \Omega_N)$ with $N$ satisfying Assumption \ref{ass_basis}, we assume the covariance matrix $\Omega_N$ satisfies $\lambda_{\min}(\Omega_N) \succsim (N/n)$.
\end{assumption}
Now we are ready to state the concentration result of the posterior center $\mu_N$. 
\begin{proposition}\label{prop:mean}
Under Assumption \ref{ass_basis} and Assumption \ref{prior_cov}, for the truncated normal posterior $\m N_{\m C} (\mu_N, \Sigma_N)$ in \S \ref{sec:band} of the main manuscript, with at least probability $1-2\,N^{-2}$ with respect to $\bbP_{Y\mid X}$, we have
$$ \| \mu_N\|_{\infty} \le \epsilon_N,$$
for $\epsilon_N \ge 2\,(c_2/c^2_1)^{1/2}\,(N\log N/n)^{1/2}$ with $c_1,c_2$ defined in Assumption \ref{ass_basis}.
\end{proposition}

\begin{proof}

Under model \eqref{model_tmvn} with true coefficient vector $\theta_0 = {\bf 0}$ we have $Y \mid \theta_0,X \sim \m N({\bf 0}_n,I_n)$. We henceforth write the posterior center $\mu(Y)=\Sigma_N\Phi^\T Y$, also we have $\bbE_{Y\mid X}\{\mu(Y)\} = \mathbf{0}$ and $\Var_{Y\mid X}\{\mu(Y)\}=\Sigma_N\Phi^\T\Phi\Sigma_N$. Further, we denote  $\sigma^2_j = \Var_{Y\mid X}\{\mu_j(Y)\}$ for $j=1,\ldots,N$. For basis matrix $\Phi$ and $\Omega_N$ satisfying Assumption \ref{ass_basis} and Assumption \ref{prior_cov} separately, we have 
\begin{align*}
c_1(n/N) \le \lambda_{\min} (\Omega_N^{-1}+\Phi^T\Phi) \le \lambda_{\max} (\Omega_N^{-1}+\Phi^T\Phi) \le (c_2+D)(n/N).
\end{align*}
Since under Assumption \ref{prior_cov}, the prior covariance matrix $\Omega_N$ satisfies $\|\Omega^{-1}_N\| \le D\, (n/N)$ for some constant $D>0$ and $\lambda_{\min}(\Omega_N^{-1}) \ge0$. Further, we have 
\begin{align}\label{variance_cond} 
\frac{c_1}{(c_1+D)^2}(N/n)\,I_N \le \Sigma_N\Phi^\T\Phi\Sigma_N \le \frac{c_2}{c^2_1}(N/n)\,I_N.
\end{align}
We define $\sigma^2_{\max} =\max_{j\le N} \{\sigma^2_j\}$, then \eqref{variance_cond} implies $\sigma^2_{\max} \le (c_2/c^2_1)(N/n)$. 
It is well known that $\max_{j\le N} |\mu_j|$ is a Lipschitz function of $\mu_j$'s with the Lipschitz constant $\sigma_{\max}$. We can also upper bound the expectation as  
$$\bbE_{Y|X}\Big(\max_{j\le N} |\mu_i|\Big) \le \{2\log(2N)\}^{1/2}\max_{j\le N}\{\sigma_j\} \le M_0 (N\log N/n)^{1/2},$$
where $M_0=2{(c_2/c^2_1)}^{1/2}$.
 
Thus we take $\epsilon_N\ge 2M_0\,(N\log N/n)^{1/2}$, we have 
\begin{align*}
\bbP_{Y\mid X}(\|\mu_N\|_{\infty}>\epsilon_N) &\le \bbP_{Y\mid X}\Big\{\Big|\max_{j\le N}|\mu_j| - \bbE_{Y\mid X}\Big(\max_{i\le N} |\mu_i|\Big)\Big | > \epsilon_N - \bbE_{Y\mid X}\Big(\max_{i\le N} |\mu_i|\Big)  \Big\} \\ \nonumber
& \le\bbP_{Y\mid X}\Big\{\Big|\max_{j\le N}|\mu_j| - \bbE_{Y\mid X}\Big(\max_{i\le N} |\mu_i|\Big) \Big| > \epsilon_N/2 \Big\}\\\nonumber
&\le 2 \exp\{-\epsilon_N^2/(8\,\sigma^2_{\max})\} \le  2\,N^{-2}.
\end{align*}
Then we have established the result.
\end{proof}

\section{Auxiliary results}\label{sec:appd}
\begin{lemma}\label{slepian}
(Slepian's lemma) Let $X, Y$ be centered Gaussian vectors on $\mathbb{R}^d$. Suppose $\bbE X_i^2 = \bbE Y_i^2$ for all $i$, and $\bbE (X_iX_j) \le \bbE (Z_iZ_j)$ for all $i \ne j$. Then, for any $x \in \mb R$,
\begin{align*}
\bbP\Big(\max_{1\le i\le d} X_i \le x\Big) \le \bbP\Big(\max_{1\le i\le d} Y_i \le x\Big).
\end{align*}
\end{lemma}
We use the Slepian's lemma in the following way in the main document. We have, 
\begin{align*}\label{eq:centered}
\bbP(X_1 \ge 0, \ldots, X_d \ge 0) = \bbP\Big(\min_{1\le i \le d} X_i \ge 0\Big) = \bbP\Big(\max_{1 \le i \le d} X_i \le 0\Big),
\end{align*}
where the second equality uses $X \overset{d}= - X$. We use Slepian's inequality to arrive at equation \eqref{eq:Slep} in the main document. 

\begin{lemma}\label{lem:l3}
Let $Z_1, \ldots, Z_N$ be iid $\m N(0, 1)$ random variables. Then we have 
\begin{align}
C_1 \sqrt{2\log N} \le \bbE \max_{i=1,\dots,N} Z_i \le \bbE \max_{i=1,\dots,N} |Z_i| \le \sqrt{2 \log N}.
\end{align}
for some constant $0< C_1<1$. 
\end{lemma}

\begin{lemma}\label{lem:l4}
(Mill's ratio bound) Let $X \sim \m N(0,1)$. We have, for $x>0$, that
\begin{align*}
\frac{x}{x^2+1} \mbox{e}^{-x^2/2} \le 1-\Phi(x) \le \frac{1}{x}\mbox{e}^{-x^2/2}, 
\end{align*} 
where $\Phi(\cdot)$ is cumulative distribution function of $X$.
\end{lemma}

 
\begin{lemma}\label{band_inv} (Lemma 2.1 in \cite{bickel2012approximating})
Let matrix $A$ be $k$-banded, symmetric, and positive definite. We denote $M = \|A\|$ and $m=1/\|A^{-1}\|$, and for $n \in \mathbb{N}_0$, we define 
\begin{align}\label{inverse}
B_n = \gamma \sum_{j=0}^n (I-\gamma A)^j, 
\end{align}
where $\gamma = 2/(M+m)$. Then $B_n$ is a symmetric positive definite $(nk)$-banded matrix, also, $\|A^{-1} - B_n \| \le \kappa^{n+1}/m$, $\kappa = (M-m)/(M+m)<1$. 
\end{lemma}

\bibliographystyle{plainnat}
\bibliography{ref1,tmvn_ref, proton}

\end{document}